\documentclass{article}
\usepackage[numbers,sort&compress]{natbib}
\let\parencite\citep
\let\textcite\citet

\usepackage[margin=1in]{geometry}
\usepackage{graphicx}
\usepackage{amsmath, amssymb, amsthm, mathtools}
 \usepackage{enumitem} 
\usepackage{algorithm}
\usepackage{algpseudocode}
\usepackage{float}
\usepackage[section]{placeins}
\usepackage[title,titletoc]{appendix}
\usepackage{booktabs}
\usepackage{tabularx}
\usepackage{threeparttable}
\usepackage{caption}
\usepackage{authblk}
\usepackage{xcolor}
\usepackage{array}
\usepackage{tikz}
\usepackage{tocloft}
\usepackage{titlesec}
\usepackage{subcaption}
\usepackage{siunitx}
\newcommand{\sciRange}[3]{%
  \ensuremath{#1\text{--}#2 \times 10^{#3}}%
}

\newcommand{\sci}[2]{%
  \ensuremath{#1 \times 10^{#2}}%
}

\newtheorem{definition}{Definition}
\newtheorem{theorem}{Theorem}
\newtheorem{remark}{Remark}

\newtheorem{proposition}{Proposition}

\newtheorem{example}{Example}
\newtheorem{assumption}{Assumption}

\newcommand{\st}[1]{#1}
\DeclareMathOperator*{\argmin}{arg\,min}

\title{Beyond Invariant Dictionary: Data-Driven Koopman Spectral Recovery with Filtered Extended Dynamic Mode Decomposition}
\author[1]{Siji Chen}
\author[2,3]{Igor Mezi{\'c}}
\author[3]{Sui Tang}
\affil[1]{Program in Applied \& Computational Mathematics, Princeton University
(\texttt{siji@princeton.edu})}
\affil[2]{Department of Mechanical Engineering, University of California, Santa Barbara
(\texttt{mezic@ucsb.edu})}
\affil[3]{Department of Mathematics, University of California, Santa Barbara
(\texttt{suitang@ucsb.edu})}
\date{}

\usepackage{hyperref}
\hypersetup{
    colorlinks=true,
    linkcolor=blue,
    citecolor=blue,
    filecolor=magenta,
    urlcolor=cyan
}

\begin{document}

\maketitle

\begin{abstract}
The Koopman operator provides a linear framework for analyzing nonlinear
dynamical systems through their spectral properties. Extended Dynamic Mode
Decomposition (EDMD) approximates this operator from snapshot data, but
non-invariant dictionaries can introduce spurious eigenvalues.

We introduce the Projected Koopman Operator Approximation framework for
constructing Filtered EDMD operators. The framework projects the Koopman action
onto admissible dictionary subspaces that need not themselves be invariant,
while exactly preserving every represented Koopman eigenpair with nonzero
eigenvalue. For range selection, the forward--intersection chain provides a
canonical hierarchy of compatible subspaces. When the Koopman action is
injective on the dictionary space, this hierarchy connects the full dictionary
to its maximal invariant core while retaining useful intermediate models.

For projection geometry, we analyze two canonical choices: a
coordinate-orthogonal projector defined by the fixed dictionary coordinates,
which requires no function-space Gram-matrix estimate but is basis-dependent,
and a function-space orthogonal projector, which yields population EDMD at the
unfiltered level. We characterize when these choices coincide and clarify
their relations to EDMD and existing subspace-selection methods that apply
EDMD on a reduced dictionary space. Algorithmically, we develop SVD-based
procedures that construct the sampled forward--intersection  chain and implement the coordinate
projector. Under independent noiseless sampling and exact-rank
identifiability, we prove that the resulting empirical operators converge
almost surely to their population counterparts.

Experiments on a Kronecker flow, a polynomial system, and the Van~der~Pol
oscillator demonstrate reduced spectral pollution on informative dictionaries.
For Van~der~Pol, an intermediate finite-tolerance level identifies candidates
aligned with the local equilibrium reference lattice, even though the
dictionary is not constructed from the corresponding exact eigenfunctions.
The projection geometry further separates the two objects on the same certified
subspace: the measure-free coordinate projector certifies the equilibrium spectrum
independently of the sampling measure, while the $L^2(\mu)$ projector approximates the
limit-cycle spectrum.
\end{abstract}

\section{Introduction}\label{sec:intro}

Analyzing nonlinear dynamical systems is a central yet challenging
task in science and engineering, as such systems rarely admit
closed-form solutions.
The Koopman operator~\cite{koopman1931hamiltonian,koopman1932dynamical}
provides a systematic route to their analysis: acting on observables
via composition with the flow, it replaces the study of nonlinear
state-space trajectories with spectral analysis of a linear, albeit
infinite-dimensional, operator~\cite{BudisicMohrMezic2012}.
Its spectral decomposition, the Koopman Mode Decomposition
(KMD)~\cite{Mezic2005}, resolves system
behavior into intrinsic oscillation frequencies, growth and decay
rates, and coordinate representations such as eigenfunctions,
providing a coordinate-free framework for analysis, prediction, and
control of nonlinear dynamical systems~\cite{RowleyMezicEtAl2009,Mezic2013}.

Since the Koopman operator~${K}$ acts on an
infinite-dimensional function space, any computational approach must
approximate it on a finite-dimensional subspace.
Extended Dynamic Mode Decomposition
(EDMD)~\cite{williams2015data} is one of the most widely used
data-driven algorithms for this purpose: it
projects~${K}$ onto the span~$\mathcal{G}$ of a
user-chosen dictionary of observables
$\mathbf{g} = (g_1,\dots,g_n)^\top$ and estimates the corresponding
matrix from snapshot pairs~\cite{williams2015data,KordaMezic2018}.
If $\mathcal{G}$ is Koopman-invariant, the projection is exact in
the sense that every eigenvalue of the finite-dimensional EDMD
matrix is a genuine Koopman eigenvalue and the corresponding
eigenfunction is faithfully represented in~$\mathcal{G}$.
In practice, however, dictionaries are selected
heuristically, typically monomials, radial basis functions, or
delay coordinates, and there is generally no guarantee that their
span is even approximately invariant.
For such generic dictionaries the projection introduces spurious
eigenvalues, \emph{spectral pollution}, that carry no dynamical
meaning yet are indistinguishable, within the EDMD output alone,
from genuine Koopman
eigenvalues~\cite{colbrook2024rigorous,colbrook2024multiverse}.

A common strategy for mitigating spectral pollution is to identify,
exactly or approximately, Koopman-invariant subspaces contained
in~$\mathcal{G}$.
Symmetric Subspace Decomposition (SSD)~\cite{haseli2021learning}
is a representative algebraic approach: it iteratively intersects
forward- and backward-time EDMD range spaces, provably converging
to the maximal Koopman-invariant subspace
inside~$\mathcal{G}$.
Its tunable variant T-SSD~\cite{haseli2023generalizing} relaxes exact
invariance by pruning directions to achieve a prescribed
accuracy--expressiveness trade-off.
Recursive Forward--Backward EDMD
(RFB-EDMD)~\cite{RFB-EDMD} builds a nested hierarchy by
recursively removing directions with large forward--backward
inconsistency, with finite-iteration convergence guarantees for
the subspace sequence.
Koopman--Schur decomposition~\cite{drmac2024koopmanschur}
replaces potentially ill-conditioned eigenvector-based modal
decompositions by orthonormal Schur bases associated with ordered
flags of Koopman-invariant subspaces.
Principal angle decomposition~\cite{conradie2026trustworthy}
quantifies the failure of $\mathcal{G}$ to be Koopman-invariant
by principal angles between $\mathcal{G}$ and its Koopman image
$K\mathcal{G}$.
{
Across these approaches, exact or approximate invariance is the central
organizing principle. Exact subspace methods may become spectrally
uninformative when the maximal invariant core is trivial, whereas
tolerance-based variants trade invariance for a larger retained space. This
raises a broader question for mitigating spectral pollution: must the selected
range itself be invariant to preserve the Koopman spectral content already
represented in the dictionary?
}

{
To answer this question, we introduce the Projected Koopman Operator
Approximation (PKOA) framework, which separates range selection from
projection geometry. An admissible projection has a range
$\mathcal S\subseteq\mathcal G$ that contains the invariant content represented
in $\mathcal G$. Neither $\mathcal S$ need be invariant nor need the projection be orthogonal
in the function-space inner product associated with the data
measure. Nevertheless, every represented Koopman eigenpair with nonzero
eigenvalue is preserved (Proposition~\ref{prop:AP-preserve}).

Filtered EDMD realizes this principle by projecting the Koopman action into a
selected compatible range before computing its spectrum. Thus,
\emph{filtering} occurs during operator construction, rather than by screening
computed eigenvalues afterward. A canonical multilevel choice of compatible
ranges is the forward--intersection chain
}
\begin{equation}\label{eq:intro-chain}
  \mathcal S_0:=\mathcal G,\qquad
  \mathcal S_{j+1}:=\mathcal S_j\cap K\mathcal S_j .
\end{equation}
{
Here $\mathcal S_0=\mathcal G$ is the unfiltered endpoint, while
$\mathcal S_1=\mathcal G\cap K\mathcal G$ is the first filtered range,
retaining observables represented both in the dictionary and as Koopman images
of dictionary observables. Because the spaces are nested, increasing the level
imposes progressively stronger image-compatibility requirements and produces
a hierarchy of increasingly selective filters. Every filtered level contains
the maximal invariant core $\mathcal V_{\max}^{\mathcal G}$, so the
preservation guarantee holds throughout the hierarchy even though the
intermediate spaces need not be invariant and may be substantially larger than
the core. For any selected level $j\geq1$, projecting the Koopman action onto
$\mathcal S_j$ defines the corresponding Filtered EDMD operator. If
$K|_{\mathcal G}$ is injective, the chain stabilizes after finitely many steps
at $\mathcal V_{\max}^{\mathcal G}$.
}

\paragraph{Contributions.}
The main results of this paper are as follows.
\begin{enumerate}
\item[(i)]
{
  \textbf{Operator-theoretic framework and canonical projectors.}
  We formalize the PKOA admissibility conditions and prove the resulting
  preservation guarantee for represented nonzero Koopman eigenpairs. For a
  fixed admissible range, we analyze two canonical choices: orthogonal
  projection in the fixed dictionary-coordinate geometry and orthogonal
  projection in the function-space geometry underlying population EDMD. We
  characterize when these choices coincide and clarify their population-level
  relations to EDMD and the reduced EDMD step used by SSD and T-SSD
  (Section~\ref{sec:operator-framework}).
}

\item[(ii)]
{
\textbf{Data-driven coordinate projection.}
We design SVD-based algorithms that construct the sampled forward--intersection
chain and its coordinate-orthogonal projector at any selected level directly
from snapshot data
(Algorithms~\ref{alg:iter-intersect} and~\ref{alg:filtered-edmd}). The
coordinate-filtering step uses Euclidean geometry on coefficient vectors and
therefore requires no estimate of a function-space Gram matrix, although it
depends on the chosen dictionary coordinates. Algorithm~\ref{alg:filtered-edmd}
recovers the maximal coefficient subspace compatible with the selected intersection
(Proposition~\ref{prop:maximal-projection}), and a least-squares step assembles
the filtered matrix.
  Under independent noiseless sampling and exact-rank identifiability, we
  prove almost-sure convergence of the empirical coordinate-projected
  Filtered EDMD operator to its population counterpart
  (Theorem~\ref{thm:convergence}) and extend the result to the full chain
  (Remark~\ref{rmk:full-chain-convergence}).
}

\item[(iii)] {{
\textbf{Numerical validation.}
Comparisons with EDMD and existing subspace- and residual-based methods on
three systems demonstrate the benefit of retaining intermediate filtered
levels (Section~\ref{sec:experiments}). On the Kronecker and polynomial
examples, these levels remove detected spurious eigenvalues while retaining
the represented spectral content. For Van~der~Pol, where the terminal
invariant model is uninformative, an intermediate finite-tolerance level
identifies candidates aligned with the local equilibrium reference lattice.
The projection geometry further distinguishes equilibrium from limit-cycle content
on the same certified subspace: the measure-free coordinate projector certifies the
equilibrium spectrum independently of the sampling measure, while the $L^2(\mu)$
projector approximates the limit-cycle spectrum (Section~\ref{sec:vdp-geometry}).}}
\end{enumerate}

\paragraph{Other related work.}
{
Residual DMD
(ResDMD)~\cite{colbrook2024rigorous,colbrook2023resdmd_fluid}
takes a complementary output-certification approach. In the eigenpair-screening
implementation used here, it retains EDMD candidates whose computable residuals
meet a prescribed threshold. Filtered EDMD instead changes the projection
before the eigendecomposition.
}
Beyond eigenvalue screening and subspace extraction, several other lines of work address spectral pollution.
For measure-preserving systems, moment-based spectral methods~\cite{korda2020data}
establish convergence guarantees for Koopman spectral measures from trajectory
moments, while delay-coordinate methods~\cite{das2019delay} approximate Koopman eigenvalues and eigenfunctions with the same guarantees.
mpEDMD~\cite{colbrook2023mpedmd} enforces the unitary structure directly at the matrix level, yielding pollution-free spectral convergence for measure-preserving dynamics. Rigged DMD~\cite{colbrook2025rigged} treats spectra beyond the
pure-point setting by passing to a rigged Hilbert space, where
discrete and continuous spectral components can be represented in
a unified spectral decomposition.
On the dictionary-design side, analytic
EDMD~\cite{mauroy2024analytic} exploits RKHS structure near
hyperbolic equilibria to avoid spectral pollution for analytically
conjugate dynamics, kernel EDMD~\cite{KernelEDMD} works in
reproducing kernel Hilbert spaces with quantitative error
bounds~\cite{philipp2025kernel_edmd,kohne2025kernel_linf}, and
neural approaches~\cite{DeepKoopman,OttoRowley2019} learn dictionaries end-to-end.
{These methods enrich the dictionary or constrain the approximation, whereas
PKOA changes the projection for a fixed dictionary.}

\paragraph{Outline.}
Section~\ref{sec:prelim}  fixes notation and collects preliminaries on Koopman operator
theory.
Section~\ref{sec:operator-framework} develops the operator-theoretic
framework.
Section~\ref{sec:algorithms} presents the algorithms and convergence
theory.
Section~\ref{sec:experiments} contains numerical experiments.
Section~\ref{sec:conclusion} concludes.

\section{Preliminaries}\label{sec:prelim}

\paragraph{Notation.}
Bold lowercase (e.g.\ $\mathbf{x},\mathbf{v}$) denote vectors and bold uppercase 
(e.g.\ $\mathbf{X},\mathbf{A}$) denote matrices.  
$\mathbf{I}_n$ is the $n\times n$ identity, $\mathbf{A}^*$ the conjugate transpose,
and $\mathbf{A}^\dagger$ the Moore--Penrose pseudoinverse.  
For functions $f,g$ on the state space $\mathcal M$ with measure $\mu$,
\[
\langle f,g\rangle_\mu := \int_\mathcal{M} f(\mathbf{v})\,\overline{g(\mathbf{v})}\,d\mu(\mathbf{v}),
\qquad
L^2(\mu):=\{ f:\mathcal M\to\mathbb C \;\mid\; \|f\|_\mu^2=\langle f,f\rangle_\mu<\infty\}.
\]
Calligraphic letters (e.g.\ $\mathcal{G},\mathcal{V}$) denote subspaces of functions.
When a finite collection $\mathbf g=(g_1,\dots,g_n)^\top$ is fixed, we also write 
$\langle \mathbf g(\cdot),\mathbf c\rangle$ for the finite-dimensional inner product 
$\sum_{j=1}^n \overline{c_j}\,g_j(\cdot)$; context will always distinguish this from the 
$L^2(\mu)$ inner product $\langle\cdot,\cdot\rangle_\mu$.

For a linear operator $T:\mathcal U\to\mathcal U$ and a subspace $\mathcal V\subseteq \mathcal U$,
$T|_{\mathcal V}:\mathcal V\to T(\mathcal V)$ denotes its restriction.  
For an operator $T$ (or matrix $\mathbf T$), $\sigma(T)$ (resp.\ $\sigma(\mathbf T)$) denotes its spectrum.
We take
\(\mathbb N=\{0,1,2,\ldots\}\).

\subsection{Koopman Operator and Eigenfunctions}

Consider a discrete-time dynamical system
\[
\mathbf{v}_{j+1}=F(\mathbf{v}_{j}), \qquad j=0,1,2,\dots,
\]
with flow map $F:\mathcal{M}\to\mathcal{M}$.  
The associated \emph{Koopman operator} $K:L^2(\mu)\to L^2(\mu)$ is defined by
\[
(Kg)(\mathbf v)=g(F(\mathbf v)).
\]
A nonzero $\psi\in L^2(\mu)$ is a \emph{Koopman eigenfunction} with eigenvalue $\lambda\in\mathbb C$ if
$K\psi = \lambda \psi$.
Eigenvalues encode temporal growth, decay, or oscillations, while eigenfunctions yield intrinsic coordinates for the dynamics.

\paragraph{Koopman Mode Decomposition.}
If an observable $g$ admits an expansion $g = \sum_i b_i\,\psi_i$ 
in Koopman eigenfunctions $\{\psi_i\}$ with eigenvalues $\{\lambda_i\}$, 
then the coefficients $\{b_i\}$ are the associated \emph{Koopman modes} and the time evolution takes the form
\[
g(\mathbf v_t)=K^t g(\mathbf v_0) = \sum_i b_i\,\lambda_i^t\,\psi_i(\mathbf v_0).
\]
This \emph{Koopman Mode Decomposition (KMD)} \parencite{Mezic2005} links spectral properties to temporal dynamics.

\subsection{Extended Dynamic Mode Decomposition}\label{sec:EDMD}

The Koopman operator acts on an infinite-dimensional function space, so practical computation requires a finite-dimensional surrogate built from data.  
Extended Dynamic Mode Decomposition (EDMD) \parencite{williams2015data} provides the standard 
construction: one selects a dictionary of observables $\mathbf g=(g_1,\dots,g_n)^\top$ spanning 
a subspace $\mathcal G := \operatorname{span}\{g_1,\dots,g_n\}\subset L^2(\mu)$, 
and approximates $K$ on $\mathcal G$ via least squares.

Given samples $\{\mathbf x_i\}_{i=1}^m\sim\mu$, define data matrices
\[
\mathbf X = \big[\mathbf g(\mathbf x_1)\;\cdots\;\mathbf g(\mathbf x_m)\big],\qquad
\mathbf Y = \big[K\mathbf g(\mathbf x_1)\;\cdots\;K\mathbf g(\mathbf x_m)\big].
\]
The EDMD matrix is the least-squares solution
\[
\mathbf A_{E,m} := \argmin_{\mathbf A\in\mathbb C^{n\times n}} \| \mathbf Y - \mathbf A \mathbf X\|_F^2
= \mathbf Y\mathbf X^\dagger.
\]
Let $A_{E,m}:\mathcal G\to\mathcal G$ denote the operator induced by this
matrix,
\[
A_{E,m}\langle\mathbf g,\mathbf c\rangle
:=\langle\mathbf A_{E,m}\mathbf g,\mathbf c\rangle.
\]
This has a variational interpretation: $\mathbf A_{E,m}$ minimizes
\[
\sum_{j=1}^n
\bigl\|P^{\hat\mu_m}_\mathcal G(Kg_j-Ag_j)\bigr\|_{\hat\mu_m}^2,
\]
where $Ag_j := \sum_{k=1}^n A_{jk}\,g_k\in\mathcal G$
(so $\mathbf A$ represents a linear operator $\mathcal G\to\mathcal G$),
$\hat\mu_m$ is the empirical measure, and $P^{\hat\mu_m}_\mathcal G$
is the $L^2(\hat\mu_m)$-orthogonal projection onto $\mathcal G$.
Under $\mu$-independence of the dictionary 
(i.e., no nontrivial linear combination of the $g_j$ vanishes $\mu$-a.e.), 
the associated operator converges as $m\to\infty$ to
\begin{equation}\label{eq:EDMD-pop}
A_E = P_\mathcal G^\mu \circ (K|_\mathcal G),
\end{equation}
the $L^2(\mu)$-orthogonal projection of $K$ onto $\mathcal G$ \parencite{KordaMezic2018}.
The pair $(\mathbf g, \mathbf A_{E,m})$ serves as a data-driven \emph{approximate linear representation} 
of the dynamics—a concept we formalize next.

\subsection{Finite-Dimensional Representations and Invariant Subspaces}\label{sec:fin-dim-rep}

The quality of EDMD hinges on the relationship between the dictionary $\mathcal G$ and 
the Koopman operator.  
In the ideal case, $\mathcal G$ is \emph{Koopman-invariant}: $K\mathcal G \subseteq \mathcal G$.  
When this holds, $K|_\mathcal G$ admits an \emph{exact} matrix representation.

\begin{definition}[Linear representation]\label{def:linear-rep}
If $K\mathcal G \subseteq \mathcal G$, there exists a unique $\mathbf A\in\mathbb C^{n\times n}$ such that
\[
K\mathbf g(\cdot) = \mathbf A\,\mathbf g(\cdot).
\]
We call $(\mathbf g, \mathbf A)$ a \emph{linear representation} of $K$ on $\mathcal G$.  
When $\mathcal G$ is not invariant, the data-driven pair $(\mathbf g, \mathbf A_{E,m})$ is called 
an \emph{approximate linear representation}.
\end{definition}

The spectral content of $\mathbf A$ directly encodes Koopman spectral objects. The following spectral correspondence is well established
\parencite{AMSNoticesKoopman}.
\begin{proposition}[Spectral correspondence]\label{prop:spectral-correspondence}
Let $(\mathbf g, \mathbf A)$ be a linear representation with $\mathbf A$ diagonalizable, 
having eigenvalues $\{\lambda_i\}$, left eigenvectors $\{\mathbf l_i\}$, 
and right eigenvectors $\{\mathbf r_i\}$, {chosen such that \(\langle \mathbf r_j,\mathbf l_i\rangle=\delta_{ij}.\)} Then:
\begin{enumerate}[label=\textnormal{(\roman*)}]
    \item The eigenvalues $\{\lambda_i\}$ of $\mathbf A$ are Koopman eigenvalues.
    \item The functions $\psi_i(\cdot):=\langle \mathbf g(\cdot),\mathbf l_i\rangle$ are Koopman eigenfunctions 
          with $K\psi_i = \lambda_i\psi_i$.
    \item The right eigenvectors $\{\mathbf r_i\}$ are Koopman modes, giving the expansion 
          $\mathbf g(\cdot)=\sum_i \mathbf r_i \psi_i(\cdot)$.
\end{enumerate}
When $\mathbf A$ is not diagonalizable, an analogous correspondence holds with generalized 
eigenfunctions (see Appendix~\ref{App:generalized_eigenfunction}).
\end{proposition}

While any finite-dimensional Koopman-invariant subspace admits such a spectral correspondence, some invariant subspaces are more informative than others.
Subspaces that contain directions mapped to zero by $K$ (e.g., observables supported on 
wandering sets) contribute degenerate spectral content with no dynamical meaning.  
This issue is absent for invertible dynamics, since the Koopman operator is then injective.
To exclude such redundancies, we restrict attention to subspaces on which $K$ acts invertibly.

\begin{definition}[Strong Koopman invariance]\label{def:strong-invariance}
A subspace $\mathcal V\subset L^2(\mu)$ is \emph{strongly Koopman-invariant} if 
$K\mathcal V=\mathcal V$ and $K|_\mathcal V$ is invertible.
\end{definition}

When \(K|_V\) has a nontrivial kernel, its elements are observables annihilated in one time step and therefore do not contribute nonzero Koopman spectral components.
Requiring invertibility ensures that the invariant subspace captures only dynamically meaningful spectral content.
\emph{Throughout this paper, ``Koopman-invariant'' means strongly Koopman-invariant 
unless otherwise stated.}

\paragraph{The spectral pollution problem.}
The invariance discussion above pinpoints precisely why EDMD struggles.
{
In formula~\eqref{eq:EDMD-pop}, the operator $A_E = P_\mathcal G^\mu \circ (K|_\mathcal G)$ 
uses the $L^2(\mu)$-orthogonal projection onto $\mathcal G$.  
When $\mathcal G$ is invariant, $P_\mathcal G^\mu$ acts as the identity on $K\mathcal G$ 
and $A_E$ recovers $K|_\mathcal G$ exactly.  
When $\mathcal G$ is \emph{not} invariant, $K\mathcal G$ extends beyond $\mathcal G$,
and the orthogonal projection can mix invariant components with transient
directions that lie outside $\mathcal G$.
This mixing can produce spurious eigenvalues, \emph{spectral pollution}, whose relation to the true Koopman spectrum is not guaranteed.
}
The spectral pollution considered here is a finite-dimensional compression effect and should be distinguished from the intrinsic limitations of the ordinary \(L^2(\mu)\) point spectrum.
{Systems with continuous spectra can instead be studied using spectral measures~\cite{korda2020data}, pseudospectral methods~\cite{colbrook2024rigorous}, or formulations in a rigged Hilbert space~\cite{colbrook2025rigged}.}
Our focus is different: we address spurious
eigenvalues created by the \(L^2(\mu)\)-orthogonal compression
\(P_{\mathcal G}^{\mu}\circ(K|_{\mathcal G})\) when \(\mathcal G\) is not
Koopman-invariant.

\section{A Projected Koopman Approximation Framework}\label{sec:operator-framework}

\subsection{The Projected Koopman Operator Approximation (PKOA)}

{As observed in Section~\ref{sec:EDMD}, the EDMD operator $A_E = P_{\mathcal{G}}^{\mu}\circ (K|_{\mathcal{G}})$ is built from the $L^2(\mu)$-orthogonal projection onto the full dictionary $\mathcal{G}$. When $\mathcal{G}$ is not Koopman-invariant, this projection can pollute the spectrum by mixing genuine invariant content with transient directions. But it is invariant subspaces, not $L^2(\mu)$-orthogonality, that anchor the Koopman spectrum. This motivates replacing the orthogonal projection with a more general one that retains the invariant structure within $\mathcal{G}$.}

\paragraph{Generalized projections.}
Let $\mathcal D_P\subseteq L^2(\mu)$ be a linear subspace, and consider a map
\[
P:\mathcal D_P\to\mathcal D_P.
\]
We call $P$ an \emph{admissible generalized projection} relative to
$\mathcal G$ if it satisfies the following conditions:
\begin{enumerate}
\item[(C1)] \emph{Linearity and idempotence:}
$P$ is linear and $P^2=P$.

\item[(C2)] \emph{Range condition:}
\[
\operatorname{ran}(P)=\mathcal G_P\subseteq\mathcal G.
\]

\item[(C3)] \emph{Domain condition:}
\[
K\mathcal G\subseteq\mathcal D_P.
\]

\item[(C4)] \emph{Invariant-content condition:}
every Koopman-invariant subspace of $\mathcal G$ is contained in
$\mathcal G_P$.
\end{enumerate}
Conditions~\textnormal{(C1)} and~\textnormal{(C2)} make $P$ an algebraic
projection onto $\mathcal G_P$; no orthogonality is required.
Condition~\textnormal{(C3)} ensures that the associated projected Koopman
approximation
\[
A_P:=P\circ(K|_{\mathcal G})
:\mathcal G\longrightarrow\mathcal G_P
\]
is well defined, while condition~\textnormal{(C4)} ensures that the projection
range retains all Koopman-invariant content represented in the dictionary.

In particular, every admissible range satisfies
\[
\mathcal V_{\max}^{\mathcal G}
\subseteq\mathcal G_P
\subseteq\mathcal G,
\]
so these ranges interpolate between the maximal invariant core and
the full dictionary. At the full-dictionary endpoint, choosing the
function-space orthogonal projector gives population EDMD. At the
invariant-core endpoint, the restriction of $A_P$ to
$\mathcal V_{\max}^{\mathcal G}$ is the exact Koopman restriction. Different
admissible projections, even those with the same range, may still
act differently away from the invariant core.

When discussing its spectrum, we regard $A_P$ as an endomorphism of
$\mathcal G$ through the natural inclusion
$\mathcal G_P\hookrightarrow\mathcal G$.

\paragraph{Spectral preservation.}
Admissible generalized projections preserve every Koopman eigenpair
in $\mathcal G$ with nonzero eigenvalue, as the following proposition
makes precise.

\begin{proposition}[Preservation of Koopman eigenpairs with nonzero eigenvalue]
\label{prop:AP-preserve}
Suppose $P$ satisfies \textnormal{(C1)--(C4)}. If
$f\in\mathcal G\setminus\{0\}$ and
\[
Kf=\lambda f,
\qquad \lambda\ne0,
\]
then $f\in\operatorname{ran}(P)$ and
\[
A_Pf=\lambda f.
\]
Consequently,
\[
\{\lambda\in\mathbb C\setminus\{0\}:
  \exists f\in\mathcal G\setminus\{0\},\ Kf=\lambda f\}
\subseteq\sigma(A_P).
\]
\end{proposition}

\begin{proof}
Since $\lambda\ne0$, the space $\operatorname{span}\{f\}$ is
Koopman-invariant and is therefore contained in $\operatorname{ran}(P)$ by
\textnormal{(C4)}. By \textnormal{(C1)}, $Pf=f$, and hence
\[
A_Pf=P(Kf)=P(\lambda f)=\lambda f.
\]
\end{proof}

\begin{remark}[Compatibility of the factorization]\label{rmk:universality}
Suppose $K|_{\mathcal G}$ is injective and a linear map
$A:\mathcal G\to\mathcal G$ is prescribed. The relation
$\widetilde P(Kf):=Af$ defines a linear map from $K\mathcal G$ into
$\operatorname{ran}(A)$. It extends to a projection onto
$\operatorname{ran}(A)$ if and only if
\[
\widetilde P y=y
\quad\text{for every }y\in K\mathcal G\cap\operatorname{ran}(A),
\]
or equivalently, $Af=Kf$ whenever $Kf\in\operatorname{ran}(A)$. If
$K|_{\mathcal G}$ is not injective, one additionally needs
$A|_{\ker(K|_{\mathcal G})}=0$ for $\widetilde P$ to be well defined. Thus,
factorization through a projection imposes a genuine compatibility condition;
(C4) further ensures preservation of the invariant content.
\end{remark}

\subsection{Filtered EDMD Operators via Forward Intersections}\label{sec:F-EDMD}

We now turn to the specific projection used throughout the paper. The construction proceeds in two stages: we first identify a natural family of subspaces compatible with the Koopman action, then define the Filtered EDMD operator by projecting onto a member of that family.

\subsubsection{The forward--intersection chain}

Any $K$-invariant subspace of $\mathcal{G}$ must be stable under one Koopman step, hence
\[
\mathcal{V} \subseteq \mathcal{G} \cap K\mathcal{G}
\]
for every $K$-invariant $\mathcal{V}\subseteq\mathcal{G}$.  
This observation suggests constructing admissible projection ranges 
by enforcing multi-step persistence under the Koopman action.
Starting from $\mathcal{G}$, we iteratively intersect with its forward image:
\begin{equation}\label{eq:chain}
\mathcal S_0 := \mathcal G,\qquad 
\mathcal S_{j+1} := \mathcal S_j \cap K\mathcal S_j\quad (j\ge 0).
\end{equation}
This generates a nested chain
\[
\mathcal G = \mathcal S_0 \supseteq \mathcal S_1 \supseteq \mathcal S_2 \supseteq \cdots
\]
that progressively filters out directions failing to persist under repeated 
Koopman actions.  The chain converges to the largest invariant subspace:

\begin{definition}[Largest $K$-invariant subspace]\label{def:largest-invariant}
Let $\mathcal{G}\subseteq L^2(\mu)$ be finite-dimensional.  
The \emph{largest $K$-invariant subspace} of $\mathcal{G}$, denoted 
$\mathcal{V}_{\max}^{\mathcal{G}}$, is the unique $K$-invariant subspace of $\mathcal{G}$ 
that contains every other $K$-invariant subspace of $\mathcal{G}$.
\end{definition}

\begin{proposition}[Fixed point of the forward--intersection chain]\label{prop:fixpoint}
Let $K:\mathcal{H}\to\mathcal{H}$ be linear on a Hilbert space $\mathcal{H}$, 
and let $\mathcal{S}_0\subset\mathcal{H}$ be finite-dimensional.  
Define the chain $\{\mathcal{S}_j\}$ by~\eqref{eq:chain}.  Then:
\begin{enumerate}
\item \textbf{Finite-step stabilization.}  
The chain stabilizes in at most $\dim \mathcal{S}_0$ steps: 
there exists $i\le \dim \mathcal{S}_0$ such that
\(
\mathcal{S}_{i+1} = \mathcal{S}_i =: \mathcal{S}_\ast,
\)
and $\mathcal{S}_\ast$ is $K$-invariant.

\item \textbf{Maximality.}  
If $K|_{\mathcal{S}_0}$ is injective, then 
$\mathcal{S}_\ast = \mathcal{V}_{\max}^{\mathcal{S}_0}$.
\end{enumerate}
\end{proposition}

\begin{proof}
(1) The sequence $(\mathcal{S}_j)$ is nested.  
If $\mathcal{S}_{i+1}=\mathcal{S}_i$, then 
$\mathcal{S}_i\subseteq K\mathcal{S}_i$; since 
$\dim(K\mathcal{S}_i)\le \dim(\mathcal{S}_i)$, equality holds, 
so $K\mathcal{S}_i=\mathcal{S}_i$.  
If no stabilization occurs, each step strictly decreases the dimension, 
which can happen at most $\dim \mathcal{S}_0$ times.

(2) Let $U\subseteq \mathcal{S}_0$ be $K$-invariant.  
Injectivity implies $\dim(KU)=\dim(U)$, so $KU=U$.  
In particular $U\subseteq K\mathcal{S}_0$, hence 
$U\subseteq \mathcal{S}_0\cap K\mathcal{S}_0=\mathcal{S}_1$.  
By induction, $U\subseteq \mathcal{S}_j$ for all $j$, 
so $U\subseteq \mathcal{S}_\ast$.
\end{proof}

The chain $\{\mathcal{S}_j\}$ has two key properties that make it 
the natural choice of projection ranges:
\emph{(i) Spectral faithfulness}---each $\mathcal{S}_j$ removes directions 
that fail to persist for $j$ Koopman steps, suppressing transient contamination 
in a controlled manner; and
\emph{(ii) Maximality}---for a given persistence level $j$,
the space $\mathcal{S}_j$ is the largest subspace consistent with
$j$-step persistence, so filtering is never more restrictive than necessary. {Maximality is structural, not spectral: $\mathcal V_{\max}^{\mathcal G}$ is
the largest Koopman-invariant subspace inside $\mathcal G$, but its richness is
fixed by the dynamics, not the dictionary. If $K$ has no nontrivial $L^2(\mu)$
eigenfunctions, none can be created by enlarging $\mathcal G$, so
$\mathcal V_{\max}^{\mathcal G}$ collapses to the constants. This is the case for
mixing dynamics (see~\cite{Mezic2022Numerical}), whose spectrum is continuous;
recovering it then requires leaving $L^2(\mu)$~\cite{colbrook2025rigged}.}

\subsubsection{The Filtered EDMD operator}

With the chain in hand, we define the Filtered EDMD operator by projecting
the Koopman action onto its members.

\begin{definition}[Admissible subspace and Filtered EDMD operator]\label{def:F-EDMD-operator}
We call a subspace $\mathcal S$ \emph{admissible} if
\begin{equation}\label{eq:admissible-subspace}
\mathcal V_{\max}^{\mathcal G}\ \subseteq\ \mathcal S\ \subseteq\ \mathcal G \cap K\mathcal G,
\end{equation}
that is, $\mathcal S$ contains the maximal Koopman-invariant subspace of $\mathcal G$ and lies
in the one-step compatible space $\mathcal G\cap K\mathcal G$. An admissible subspace is in
particular an admissible projection range in the sense of \textnormal{(C1)--(C4)}, with the
additional property $\mathcal S\subseteq K\mathcal G$; this inclusion makes the
coordinate-orthogonal projector of Section~\ref{sec:canonical-projectors} well defined.
For an admissible $\mathcal S$, let $P_{\mathcal S}:K\mathcal G \to \mathcal S$ be a projection
(not necessarily $L^2(\mu)$-orthogonal).
The \emph{Filtered EDMD operator} associated with $\mathcal S$ is
\begin{equation}\label{GFEDMD}
A_{\mathcal S} \;:=\; P_{\mathcal S} \circ (K|_{\mathcal G}):\ \mathcal G \longrightarrow \mathcal S.
\end{equation}
\end{definition}

This construction specializes the PKOA framework by restricting the projection range 
to subspaces that are compatible with the Koopman action: 
$\mathcal{S}$ always contains $\mathcal{V}_{\max}^{\mathcal{G}}$ while excluding 
directions outside $\mathcal{G}\cap K\mathcal{G}$ that cannot carry genuine spectral content.
The projection $P_{\mathcal S}$ need not be orthogonal, consistent with developments in nonlinear
reduced-order modeling, where oblique projections align reduced
representations with dynamically relevant
directions~\cite{OttoMacchioRowley2023}.

\begin{remark}[Role of the projection]\label{rmk:projection-role}
{
Definition~\ref{def:F-EDMD-operator} leaves the projection 
$P_{\mathcal S}$ unspecified because the spectral guarantee in
Proposition~\ref{prop:AP-preserve} depends only on its range. Any two
projections onto the same admissible $\mathcal S$ yield Filtered EDMD
operators that agree with $K$ on $\mathcal V_{\max}^{\mathcal G}$ and hence
preserve the same Koopman eigenpairs with nonzero eigenvalues, although they may differ on
complementary directions in $\mathcal G$. Thus, $\mathcal S$ determines the retained range,
while $P_{\mathcal S}$ determines the action away from the invariant core.
Section~\ref{sec:canonical-projectors} compares the coordinate-orthogonal and
$L^2(\mu)$-orthogonal choices and their relations to EDMD and the comparison
methods;
Section~\ref{sec:algorithms} constructs the coordinate-orthogonal choice from
data.
}
\end{remark}

Each filtered level $\mathcal S_j$, $j\geq1$, is an admissible projection
range because
$\mathcal V_{\max}^{\mathcal G}\subseteq\mathcal S_\ast
\subseteq\mathcal S_j\subseteq\mathcal S_1
=\mathcal G\cap K\mathcal G$. Choosing any projection
$P_{\mathcal S_j}:K\mathcal G\to\mathcal S_j$ then gives an admissible PKOA,
so $A_{\mathcal S_j}$ inherits the spectral preservation guarantee of
Proposition~\ref{prop:AP-preserve}.
{
\subsubsection{Canonical projection choices and their relation to EDMD}
\label{sec:canonical-projectors}

For this population-level comparison, assume that
$\{g_j\}_{j=1}^n$ is linearly independent in $L^2(\mu)$ and that
$K|_{\mathcal G}$ is injective. Then
$K\mathbf g=(Kg_1,\ldots,Kg_n)^\top$ is a basis of $K\mathcal G$. Following
the coefficient convention of Section~\ref{sec:EDMD}, define the
conjugate-linear bijection
\[
\Psi_K:\mathbb C^n\to K\mathcal G,
\qquad
\Psi_K\mathbf c:=\langle K\mathbf g,\mathbf c\rangle
=\sum_{j=1}^n\overline{c_j}\,Kg_j.
\]
For an admissible $\mathcal S$, set
$W_{\mathcal S}:=\Psi_K^{-1}(\mathcal S)$. This gives two natural choices for
the projector in Definition~\ref{def:F-EDMD-operator}.

\begin{enumerate}[label=\textnormal{(\roman*)}]
\item \emph{Coordinate-orthogonal projection.}
Let $r=\dim\mathcal S$ and let $\mathbf Q\in\mathbb C^{n\times r}$ have
Euclidean-orthonormal columns spanning $W_{\mathcal S}$. Define
\[
\boldsymbol\Pi_{\mathcal S}^{\mathrm c}:=\mathbf Q\mathbf Q^*,
\qquad
P_{\mathcal S}^{\mathrm c}\Psi_K\mathbf c
:=\Psi_K\boldsymbol\Pi_{\mathcal S}^{\mathrm c}\mathbf c,
\qquad
A_{\mathcal S}^{\mathrm c}
:=P_{\mathcal S}^{\mathrm c}\circ(K|_{\mathcal G}).
\]
The coefficient projector $\boldsymbol\Pi_{\mathcal S}^{\mathrm c}$ is
Euclidean-orthogonal, whereas the induced function-space projector
$P_{\mathcal S}^{\mathrm c}$ is generally oblique in $L^2(\mu)$. Once
$W_{\mathcal S}$ is fixed, this construction requires no $L^2(\mu)$ inner
products or Gram-matrix estimate; its geometry is instead determined by the
chosen $K\mathbf g$ coordinates, hence basis-dependent.

\item \emph{$L^2(\mu)$-orthogonal projection.}
Let $P_{\mathcal S}^{\mu}:L^2(\mu)\to\mathcal S$ be the orthogonal projector
and define
\[
A_{\mathcal S}^{\mu}:=P_{\mathcal S}^{\mu}\circ(K|_{\mathcal G}).
\]
This choice depends on $\mu$ through the intrinsic function-space geometry. It also extends to
any finite-dimensional target contained in $\mathcal G$, admissible or not. In particular, at $\mathcal S_0=\mathcal G$,
$A_{\mathcal G}^{\mu}=A_E$. The coordinate choice includes
$\mathcal S_0$ only when $K\mathcal G=\mathcal G$; otherwise it begins at
$\mathcal S_1\subseteq K\mathcal G$.
\end{enumerate}

Injectivity is used only to identify $K\mathbf g$ as a basis and to make the
Gram matrix below positive definite. Without injectivity, set
$\widetilde r:=\dim\Psi_K^{-1}(\mathcal S)
=\dim\mathcal S+\dim\ker\Psi_K$ and choose
$\mathbf Q\in\mathbb C^{n\times\widetilde r}$ with Euclidean-orthonormal
columns spanning $\Psi_K^{-1}(\mathcal S)$. Since
$\ker\Psi_K\subseteq\Psi_K^{-1}(\mathcal S)$, the projector fixes
$\ker\Psi_K$, so $P_{\mathcal S}^{\mathrm c}\Psi_K\mathbf c$ is independent
of the coefficient representative. Using this $\mathbf Q$ in the same formulas
gives the coefficient-space extension used by the algorithms below.

\begin{proposition}[Canonical projectors and EDMD]
\label{prop:canonical-projectors-edmd}
Let $\mathbf H_K$ be the positive-definite Gram matrix satisfying
\[
\langle\Psi_K\mathbf a,\Psi_K\mathbf b\rangle_\mu
=\mathbf a^*\mathbf H_K\mathbf b.
\]
For each admissible $\mathcal S$,
\begin{equation}\label{eq:canonical-projector-coincidence}
P_{\mathcal S}^{\mathrm c}=P_{\mathcal S}^{\mu}|_{K\mathcal G}
\quad\Longleftrightarrow\quad
\mathbf H_K\boldsymbol\Pi_{\mathcal S}^{\mathrm c}
=\boldsymbol\Pi_{\mathcal S}^{\mathrm c}\mathbf H_K.
\end{equation}
Consequently, when these equivalent conditions hold,
$A_{\mathcal S}^{\mathrm c}=A_{\mathcal S}^{\mu}$. In particular, the two
operators coincide for every admissible $\mathcal S$ when $K\mathbf g$ is
$L^2(\mu)$-orthonormal. Mutually orthogonal elements with a common norm are
also sufficient; orthogonality with unequal norms is not sufficient for an
arbitrary $\mathcal S$.

For any finite-dimensional $\mathcal U\subseteq L^2(\mu)$, define
$A_E^{\mathcal U}:=P_{\mathcal U}^{\mu}\circ(K|_{\mathcal U})$. Then every
finite-dimensional $\mathcal S\subseteq\mathcal G$ satisfies
\begin{equation}\label{eq:l2-filtered-edmd-relation}
A_{\mathcal S}^{\mu}=P_{\mathcal S}^{\mu}\circ A_E^{\mathcal G},
\qquad
A_{\mathcal S}^{\mu}|_{\mathcal S}=A_E^{\mathcal S},
\qquad A_E^{\mathcal G}=A_E.
\end{equation}
\end{proposition}

\begin{proof}
Transporting the $L^2(\mu)$ inner product to $\mathbb C^n$ gives the metric
$\mathbf H_K$. The Euclidean projector
$\boldsymbol\Pi_{\mathcal S}^{\mathrm c}$ is orthogonal in this metric
exactly when $W_{\mathcal S}$ reduces $\mathbf H_K$, which is equivalent
to~\eqref{eq:canonical-projector-coincidence}. Transporting this identity back
through $\Psi_K$ gives the asserted projector equality. For
\eqref{eq:l2-filtered-edmd-relation}, the inclusion
$\mathcal S\subseteq\mathcal G$ implies
$P_{\mathcal S}^{\mu}\circ P_{\mathcal G}^{\mu}=P_{\mathcal S}^{\mu}$.
Composing with $K|_{\mathcal G}$ gives the first identity, and restricting to
$\mathcal S$ gives the second.
\end{proof}

\begin{example}[The two projections and EDMD]\label{ex:filtered-vs-orthogonal}
Let $F(x)=x^2$ on $[0,1]$, let $\mu$ be Lebesgue measure, and let
$Kf=f\circ F$. Take $\mathcal G=\operatorname{span}\{1,x\}$. Then
$K\mathcal G=\operatorname{span}\{1,x^2\}$ and
$\mathcal S=\mathcal G\cap K\mathcal G=\operatorname{span}\{1\}$. On
$K\mathcal G$,
\[
P_{\mathcal S}^{\mathrm c}(a+bx^2)=a,
\qquad
P_{\mathcal S}^{\mu}(a+bx^2)=a+\frac b3.
\]
Therefore
\[
A_{\mathcal S}^{\mathrm c}1=A_{\mathcal S}^{\mu}1=A_E1=1,
\qquad
A_{\mathcal S}^{\mathrm c}x=0,
\qquad
A_{\mathcal S}^{\mu}x=\frac13,
\qquad
A_Ex=x-\frac16.
\]
Thus the three operators agree on the invariant constant function but differ
on the non-invariant direction $x$.
\end{example}

\begin{remark}[Relation to comparison methods]\label{Relation_to_SSD}
At the population level, exact SSD~\citep{haseli2021learning} targets the
terminal invariant space $\mathcal V_{\max}^{\mathcal G}$, whereas
T-SSD~\citep{haseli2023generalizing} selects a tolerance-dependent subspace.
Both then fit EDMD on the selected space. Thus, if either method selects
$\mathcal U\subseteq\mathcal G$, its reduced population operator is
\[
A_E^{\mathcal U}=A_{\mathcal U}^{\mu}|_{\mathcal U}.
\]
If an admissible $\mathcal S$ is Koopman-invariant, no orthogonality assumption on
$K\mathbf g$ is needed and
\[
A_{\mathcal S}^{\mathrm c}|_{\mathcal S}
=A_{\mathcal S}^{\mu}|_{\mathcal S}
=A_E^{\mathcal S}
=K|_{\mathcal S};
\]
the reduced SSD or T-SSD operator also equals $K|_{\mathcal S}$ whenever that
method selects $\mathcal S$.
In particular, if $K\mathcal G\subseteq\mathcal G$, injectivity gives
$K\mathcal G=\mathcal G$ and exact SSD selects $\mathcal G$. If instead
$\mathcal S_1=\mathcal G\cap K\mathcal G$ is invariant,
Proposition~\ref{prop:fixpoint} gives
$\mathcal S_1=\mathcal V_{\max}^{\mathcal G}$ and exact SSD selects
$\mathcal S_1$. In either case, the reduced T-SSD operator equals the
corresponding Koopman restriction whenever T-SSD selects that space.

ResDMD and RFB-EDMD do not enter these population identities;
Section~\ref{sec:experiments} compares their finite-data outputs.

These identities concern restrictions to $\mathcal S$: Filtered EDMD acts from
$\mathcal G$ into $\mathcal S$, whereas SSD and T-SSD act on $\mathcal S$.
For any such invariant $\mathcal S$ and a decomposition
$\mathcal G=\mathcal S\oplus\mathcal T$, the zero extension of the reduced
operator equals $A_{\mathcal S}^{\diamond}$ on $\mathcal G$ if and only if
\[
P_{\mathcal S}^{\diamond}K\mathcal T=\{0\},
\qquad \diamond\in\{\mathrm c,\mu\}.
\]
\end{remark}

Below, $P_{\mathcal S}^{\mathrm c}$ and $A_{\mathcal S}^{\mathrm c}$ denote
the population function-space projector and operator. At finite sample size,
$\widehat{\mathcal S}_{j,m}$ denotes a
sample-value space, $W_{\mathcal S_j,m}$ its compatible coefficient preimage,
$\boldsymbol\Pi_{\mathcal S_j,m}^{\mathrm c}$ the corresponding coefficient
projector, $\mathbf A_{\mathcal S_j,m}^{\mathrm c}$ the empirical matrix, and
$A_{\mathcal S_j,m}^{\mathrm c}$ its induced operator on $\mathcal G$.
}

\section{Algorithms and Convergence}\label{sec:algorithms}

{
For each filtered level $j\geq1$, the population construction first selects a
target $\mathcal S_j$ along the forward--intersection chain and then projects
$K\mathcal G$ onto it. The finite-data construction mirrors these steps. From
the snapshot matrices $\mathbf X$ and $\mathbf Y$, we first construct the
sampled intersections $\widehat{\mathcal S}_{j,m}$. For a selected level, we
then recover $W_{\mathcal S_j,m}$, form the coordinate-orthogonal coefficient
projector $\boldsymbol\Pi_{\mathcal S_j,m}^{\mathrm c}$, and assemble
$\mathbf A_{\mathcal S_j,m}^{\mathrm c}$. Section~\ref{sec:convergence}
relates these finite-data objects to their population counterparts.
}

\subsection{Data-driven algorithms}\label{sec:alg-data-driven}

\paragraph{Finite-data pipeline.}
Recall the data matrices $\mathbf X, \mathbf Y \in \mathbb C^{n \times m}$
from Section~\ref{sec:EDMD}, recording evaluations of the dictionary $\mathbf{g}$ 
and its Koopman images at $m$ sample points.
Once a sampled target has been selected, let
$\boldsymbol\Pi_{\mathcal S,m}^{\mathrm c}\in\mathbb C^{n\times n}$ be its
coordinate-orthogonal projector in the $K\mathbf g$ coefficient space. The empirical Filtered EDMD matrix is
\begin{equation}\label{eq:AFm}
\mathbf A_{\mathcal S,m}^{\mathrm c}
:=\bigl(\boldsymbol\Pi_{\mathcal S,m}^{\mathrm c}\mathbf Y\bigr)
\mathbf X^\dagger.
\end{equation}
It induces an operator $A_{\mathcal S,m}^{\mathrm c}:\mathcal G\to\mathcal G$
by
\begin{equation}\label{eq:empirical-induced-operator}
A_{\mathcal S,m}^{\mathrm c}\langle\mathbf g,\mathbf c\rangle
:=\langle\mathbf A_{\mathcal S,m}^{\mathrm c}\mathbf g,\mathbf c\rangle.
\end{equation}
The matrix $\mathbf Y\mathbf X^\dagger$ is the usual EDMD least-squares
regression. Left multiplication by
$\boldsymbol\Pi_{\mathcal S,m}^{\mathrm c}$ enforces the coordinate-filtering
step.

The pair $(\mathbf{g},\mathbf A_{\mathcal S,m}^{\mathrm c})$ is an approximate
linear representation in the sense of Definition~\ref{def:linear-rep}.

\paragraph{Sampled forward--intersection chain.}
Define the sampling operator
\begin{equation}\label{eq:sampling-operator}
\mathcal E_m f := (f(\mathbf x_1),\ldots,f(\mathbf x_m))
\in\mathbb C^{1\times m},
\qquad 
\mathbf X = \mathcal E_m(\mathbf g),\quad 
\mathbf Y = \mathcal E_m(K\mathbf g).
\end{equation}
In the last two identities, $\mathcal E_m$ is applied componentwise.
The row spaces $\operatorname{row}(\mathbf X)$ and $\operatorname{row}(\mathbf Y)$ 
are the sampled counterparts of $\mathcal G$ and $K\mathcal G$.
The first sample-value intersection is
\begin{equation}\label{eq:sampled-S1}
\widehat{\mathcal S}_{1,m}
:=\operatorname{row}(\mathbf X)\cap \operatorname{row}(\mathbf Y)
\subseteq\mathbb C^{1\times m},
\end{equation}
which can be represented via a coefficient matrix $\mathbf C_1$ satisfying 
$\operatorname{row}(\mathbf C_1\mathbf X)=\widehat{\mathcal S}_{1,m}$.
More generally, $\mathbf C_j$ encodes a sample-value space through
$\operatorname{row}(\mathbf C_j\mathbf X)$, while
$\operatorname{row}(\mathbf C_j\mathbf Y)$ contains the samples of its
forward image. The recursion is
\begin{equation}\label{eq:sampled-induction}
\widehat{\mathcal S}_{j+1,m}
:=\operatorname{row}(\mathbf C_j\mathbf X)
\cap\operatorname{row}(\mathbf C_j\mathbf Y),
\qquad 
\operatorname{row}(\mathbf C_{j+1}\mathbf X)
=\widehat{\mathcal S}_{j+1,m},
\end{equation}
with $\widehat{\mathcal S}_{0,m}:=\operatorname{row}(\mathbf X)$. This produces
the nested chain of sample-value spaces
\[
\widehat{\mathcal S}_{0,m}
\supseteq\widehat{\mathcal S}_{1,m}
\supseteq\widehat{\mathcal S}_{2,m}\supseteq\cdots,
\]
the finite-data analogue of the operator-theoretic chain
$\mathcal G \supseteq \mathcal S_1 \supseteq \mathcal S_2 \supseteq \cdots
\supseteq \mathcal V_{\max}^{\mathcal G}$. Under sampling identifiability,
Proposition~\ref{prop:Em-intersections} makes this correspondence exact.
Algorithm~\ref{alg:iter-intersect} computes the coefficient matrices $\{\mathbf C_j\}$ 
via SVD, providing a numerically stable realization of these intersections.

\begin{algorithm}[!tbh]
\caption{Iterative Row-Space Intersections (Coefficient Form)}
\label{alg:iter-intersect}
\begin{algorithmic}[1]
\Statex \textbf{Input:} $\mathbf X,\mathbf Y\in\mathbb C^{n\times m}$ with $\operatorname{rank}(\mathbf X)=n$
\Statex \textbf{Output:} Coefficient matrices $\{\mathbf C_j\}_{j=1}^{J}$ s.t.\ $\operatorname{row}(\mathbf C_j\mathbf X)=\widehat{\mathcal S}_{j,m}$

\State \textbf{Initialize.} Set $\mathbf C_0 \gets \mathbf I_n$ and $r_0 \gets n$.

\For{$j=0,1,2,\ldots$}
  \State \textbf{Form stacked matrix.} 
  \(
  \mathbf Z_j :=
  \begin{bmatrix}
  \mathbf C_j\mathbf X\\[2pt]
  \mathbf C_j\mathbf Y
  \end{bmatrix}
  \in \mathbb C^{2r_j\times m}.
  \)

  \State \textbf{SVD and rank.} Compute a full SVD 
  $\mathbf Z_j=\mathbf U_j\,\boldsymbol\Sigma_j\,\mathbf V_j^*$ 
  and set $r_j' := \operatorname{rank}(\mathbf Z_j)$.

  \State \textbf{Left nullspace block.} Partition $\mathbf U_j$ conformally as
\[
\mathbf U_j =
\begin{bmatrix}
\mathbf U_j^{(11)} & \mathbf U_j^{(12)} \\[4pt]
\mathbf U_j^{(21)} & \mathbf U_j^{(22)}
\end{bmatrix},
\]
where $\mathbf U_j^{(12)} \in \mathbb C^{r_j \times (2r_j-r_j')}$ satisfies
\(
\operatorname{row}\!\big((\mathbf U_j^{(12)})^\ast\,\mathbf C_j\mathbf X\big)
= \operatorname{row}(\mathbf C_j\mathbf Y) \cap \operatorname{row}(\mathbf C_j\mathbf X).
\)

  \State \textbf{Intersection coefficients.} Compute a rank-revealing QR
  factorization (or economy SVD), and let
  $\mathbf Q_j\in\mathbb C^{r_j\times k_{j+1}}$ have orthonormal columns
  spanning $\operatorname{col}(\mathbf U_j^{(12)})$, where
  $k_{j+1}=\operatorname{rank}(\mathbf U^{(12)}_j)$.
  
  \State \textbf{Update coefficients.} 
  $\mathbf C_{j+1} \gets \mathbf Q_j^*\,\mathbf C_j \in \mathbb C^{k_{j+1}\times n}$.

  \State \textbf{Stopping rule.} If $k_{j+1}=0$ (zero intersection) or
  $k_{j+1}=r_j$ (no rank drop), set $J:=j+1$ and \textbf{break};
  otherwise set $r_{j+1}\gets k_{j+1}$.
\EndFor

\State \Return $\{\mathbf C_j\}_{j=1}^{J}$.
\end{algorithmic}
\end{algorithm}

\begin{remark}[Numerical rank]\label{rmk:numerical-rank}
In finite precision, we implement these intersections using numerical rank.
For the stacked matrix $\mathbf Z_j$, let $\{\sigma_k\}$ be its singular values.  
With tolerance $\varepsilon\in(0,1)$, set
\[
r_j' := \#\{\,\sigma_k:\sigma_k>\tau(\varepsilon)\,\},\qquad
\tau(\varepsilon)=
\begin{cases}
\varepsilon, & \text{absolute cutoff},\\
\varepsilon\cdot\sigma_{\max}, & \text{relative cutoff}.
\end{cases}
\]
Truncation at rank $r_j'$ filters small singular directions; whether a
dynamically relevant direction is retained depends on the tolerance relative
to the corresponding singular-value gap.
\end{remark}

\paragraph{Assembling the Filtered EDMD matrix.}

Given a coefficient matrix $\mathbf C$ encoding a selected sample-value target
$\widehat{\mathcal S}_m:=\operatorname{row}(\mathbf C\mathbf X)$, define its
compatible preimage in the $K\mathbf g$ coefficient space by
\begin{equation}\label{eq:empirical-coordinate-preimage}
W_{\mathcal S,m}
:=\bigl\{\mathbf c\in\mathbb C^n:
\mathbf c^*\mathbf Y\in\widehat{\mathcal S}_m\bigr\}.
\end{equation}
Algorithm~\ref{alg:filtered-edmd} recovers this space as
$\operatorname{col}(\mathbf U^{(12)})$, orthonormalizes it to obtain
$\mathbf Q$, and forms
\[
\boldsymbol\Pi_{\mathcal S,m}^{\mathrm c}:=\mathbf Q\mathbf Q^*.
\]
{
In summary, the finite-data construction follows the pipeline
\[
\widehat{\mathcal S}_m
\longrightarrow W_{\mathcal S,m}
\longrightarrow\boldsymbol\Pi_{\mathcal S,m}^{\mathrm c}
\longrightarrow\mathbf A_{\mathcal S,m}^{\mathrm c}.
\]
}

\begin{algorithm}[!tbh]
\caption{Coordinate-Projected Filtered EDMD Matrix}
\label{alg:filtered-edmd}
\begin{algorithmic}[1]
\Statex \textbf{Input:} Data matrices $\mathbf X,\mathbf Y\in\mathbb C^{n\times m}$ with $\operatorname{rank}(\mathbf X)=n$; 
coefficient matrix $\mathbf C\in\mathbb C^{r\times n}$ with
$\widehat{\mathcal S}_m:=\operatorname{row}(\mathbf C\mathbf X)
\subseteq\operatorname{row}(\mathbf Y)$, as holds for every filtered level
returned by Algorithm~\ref{alg:iter-intersect}.
\Statex \textbf{Output:} Coefficient projector
$\boldsymbol\Pi_{\mathcal S,m}^{\mathrm c}\in\mathbb C^{n\times n}$ and
Filtered EDMD matrix $\mathbf A_{\mathcal S,m}^{\mathrm c}\in\mathbb C^{n\times n}$.

\State \textbf{Stack and factorize.}  
Form
$\mathbf Z := \begin{bmatrix}\mathbf Y\\[2pt]\mathbf C\mathbf X\end{bmatrix}\in\mathbb C^{(n+r)\times m}$,
compute its full SVD $\mathbf Z=\mathbf U\boldsymbol\Sigma\mathbf V^\ast$, 
and set $r':=\operatorname{rank}(\mathbf Z)$.

\State \textbf{Extract intersection coefficients.}  
Partition $\mathbf U$ conformally as
$\mathbf U =
\bigl[\begin{smallmatrix}
\mathbf U^{(11)} & \mathbf U^{(12)} \\
\mathbf U^{(21)} & \mathbf U^{(22)}
\end{smallmatrix}\bigr]$;
the block $\mathbf U^{(12)}\in\mathbb C^{n\times(n+r-r')}$ satisfies
\[
\operatorname{col}(\mathbf U^{(12)})=W_{\mathcal S,m},
\qquad
\operatorname{row}\!\big((\mathbf U^{(12)})^\ast\mathbf Y\big)
=\widehat{\mathcal S}_m.
\]

\State \textbf{Orthogonalize.}
Compute a rank-revealing QR factorization (or economy SVD), and let
$\mathbf Q\in\mathbb C^{n\times k}$ have orthonormal columns spanning
$\operatorname{col}(\mathbf U^{(12)})$, where
$k=\dim(W_{\mathcal S,m})=\operatorname{rank}(\mathbf U^{(12)})$.

\State \textbf{Projector.}  
Set $\boldsymbol\Pi_{\mathcal S,m}^{\mathrm c}:=\mathbf Q\mathbf Q^\ast$.

\State \textbf{Assemble.}
$\mathbf A_{\mathcal S,m}^{\mathrm c}
:=(\boldsymbol\Pi_{\mathcal S,m}^{\mathrm c}\mathbf Y)\mathbf X^\dagger$.

\State \Return
$\boldsymbol\Pi_{\mathcal S,m}^{\mathrm c}$ and
$\mathbf A_{\mathcal S,m}^{\mathrm c}$.
\end{algorithmic}
\end{algorithm}

Algorithm~\ref{alg:filtered-edmd} realizes the coordinate-orthogonal choice
introduced in Section~\ref{sec:canonical-projectors}. The following proposition
shows that $\operatorname{col}(\mathbf Q)$ is the full compatible coefficient
preimage of the selected sample-value target; its maximality is a property of the
recovered range, independent of the projection geometry placed on it.

\begin{proposition}[Maximality of the recovered coefficient subspace]
\label{prop:maximal-projection}
\st{Let\/ $\mathbf C\in\mathbb C^{r\times n}$ encode a sampled subspace 
$\widehat{\mathcal S}_m=\operatorname{row}(\mathbf C\mathbf X)
\subseteq\operatorname{row}(\mathbf Y)$,
and let\/ $\mathbf U^{(12)}\in\mathbb C^{n\times(n+r-r')}$ be the block 
extracted in Step~2 of Algorithm~\ref{alg:filtered-edmd}.
Then
\[
\operatorname{col}(\mathbf U^{(12)})
\;=\;W_{\mathcal S,m}\;=\;
\bigl\{\,\mathbf v\in\mathbb C^n : 
\mathbf v^*\mathbf Y \in \operatorname{row}(\mathbf C\mathbf X)\,\bigr\}.
\]
That is, a coefficient direction $\mathbf v$ is compatible with
$\widehat{\mathcal S}_m$
if and only if\/ $\mathbf v\in\operatorname{col}(\mathbf U^{(12)})$.
Consequently,
$W_{\mathcal S,m}=\operatorname{col}(\mathbf Q)$ is the unique maximal
coefficient subspace whose directions remain in $\widehat{\mathcal S}_m$ after
application to $\mathbf Y$: if $\mathcal W\subseteq\mathbb C^n$ and
$\mathbf v^*\mathbf Y\in\widehat{\mathcal S}_m$ for every
$\mathbf v\in\mathcal W$, then
$\mathcal W\subseteq W_{\mathcal S,m}$.}
\end{proposition}

\begin{proof}
Write 
$\mathbf Z=\bigl[\begin{smallmatrix}\mathbf Y\\\mathbf C\mathbf X\end{smallmatrix}\bigr]
\in\mathbb C^{(n+r)\times m}$, 
and let 
$\mathbf U=\bigl[\begin{smallmatrix}
\mathbf U^{(11)}&\mathbf U^{(12)}\\
\mathbf U^{(21)}&\mathbf U^{(22)}
\end{smallmatrix}\bigr]$ 
be the full unitary matrix from the SVD 
$\mathbf Z=\mathbf U\boldsymbol\Sigma\mathbf V^*$, 
partitioned so that the first $n$ rows correspond to $\mathbf Y$ 
and the last $r$ rows to $\mathbf C\mathbf X$.
The columns of 
$\bigl[\begin{smallmatrix}
\mathbf U^{(12)}\\\mathbf U^{(22)}
\end{smallmatrix}\bigr]$
form an orthonormal basis for the left null space of $\mathbf Z$.

\medskip\noindent
\emph{($\subseteq$)}.\;
If $\mathbf v\in\operatorname{col}(\mathbf U^{(12)})$, 
there exists $\mathbf w\in\mathbb C^r$ such that 
$[\mathbf v;\,\mathbf w]$ lies in the left null space of $\mathbf Z$, 
i.e.\ $\mathbf v^*\mathbf Y+\mathbf w^*\mathbf C\mathbf X=\mathbf 0$.
Hence $\mathbf v^*\mathbf Y=-\mathbf w^*\mathbf C\mathbf X
\in\operatorname{row}(\mathbf C\mathbf X)$.

\medskip\noindent
\emph{($\supseteq$)}.\;
Conversely, suppose $\mathbf v\in\mathbb C^n$ satisfies 
$\mathbf v^*\mathbf Y\in\operatorname{row}(\mathbf C\mathbf X)$.
Then there exists $\mathbf w\in\mathbb C^r$ with 
$\mathbf v^*\mathbf Y=-\mathbf w^*\mathbf C\mathbf X$, 
so $[\mathbf v;\,\mathbf w]^*\mathbf Z=\mathbf 0$.
Since the left null space of $\mathbf Z$ is spanned by the columns of 
$\bigl[\begin{smallmatrix}
\mathbf U^{(12)}\\\mathbf U^{(22)}
\end{smallmatrix}\bigr]$,
there exists a coefficient vector $\boldsymbol\alpha$ such that 
$\mathbf v = \mathbf U^{(12)}\boldsymbol\alpha$, 
giving $\mathbf v\in\operatorname{col}(\mathbf U^{(12)})$.

\medskip\noindent
Since $\operatorname{col}(\mathbf Q)=\operatorname{col}(\mathbf U^{(12)})$, 
the space $\operatorname{col}(\mathbf Q)$ is precisely the set of
compatible directions. The matrix
$\boldsymbol\Pi_{\mathcal S,m}^{\mathrm c}=\mathbf Q\mathbf Q^*$ is the
Euclidean-orthogonal projector onto this maximal recovered range.
\end{proof}

This range identification, together with computational simplicity
(a single SVD followed by a rank-revealing orthogonalization),
motivates our use of the coordinate-orthogonal projector.

\begin{remark}[Projection geometry]\label{rmk:projection-choice}
By Proposition~\ref{prop:canonical-projectors-edmd}, the induced
function-space projector $P_{\mathcal S}^{\mathrm c}$ is generally oblique in
$L^2(\mu)$ and agrees with $P_{\mathcal S}^{\mu}|_{K\mathcal G}$ exactly under
the stated Gram-compatibility condition. The algorithms and convergence
results below concern the coordinate choice; implementing the $L^2(\mu)$
choice would additionally require estimating the Gram matrix of $K\mathbf g$.
\end{remark}

\begin{remark}\label{rmk:algorithm-ssd}
{
Unlike SSD, which retains only the terminal invariant space,
Algorithms~\ref{alg:iter-intersect}--\ref{alg:filtered-edmd} retain the sampled
chain and construct a coefficient projector at each filtered level; see
Remark~\ref{Relation_to_SSD}.
}
\end{remark}

\subsection{Convergence theory}\label{sec:convergence}

We now establish that the empirical Filtered EDMD operator converges to its 
population counterpart as the sample size grows.
The analysis proceeds in three stages: convergence at the first intersection, 
extension to intermediate levels, and passage to the terminal invariant subspace.
Throughout, the convergence results concern the empirical coefficient
projector $\boldsymbol\Pi_{\mathcal S_j,m}^{\mathrm c}$ constructed by
Algorithm~\ref{alg:filtered-edmd} and the induced Filtered EDMD operator
$A_{\mathcal S_j,m}^{\mathrm c}$; by
Remark~\ref{rmk:projection-role}, the spectral preservation
guarantee holds independently of this choice.

To certify that sampled intersections faithfully reproduce their operator-theoretic 
counterparts, we impose a mild identifiability condition.

\begin{assumption}[Sampling identifiability]\label{ass:sampling-identifiability}
The sampling operator $\mathcal E_m$ is injective on $\mathcal G+K\mathcal G$, i.e.,
\[
\ker(\mathcal E_m)\cap(\mathcal G+K\mathcal G)=\{0\}.
\]
Equivalently, in the noiseless setting with $m\ge\dim(\mathcal G+K\mathcal G)$,
\begin{equation}\label{eq:rank-condition}
\operatorname{rank}\!\begin{bmatrix}\mathbf Y\\ \mathbf X\end{bmatrix}
=\dim(\mathcal G+K\mathcal G).
\end{equation}
\end{assumption}

\begin{proposition}[Intersections commute with sampling]\label{prop:Em-intersections}
Under Assumption~\ref{ass:sampling-identifiability}, for any 
$\mathcal U,\mathcal V\subset \mathcal G+K\mathcal G$,
\[
\mathcal E_m(\mathcal U\cap \mathcal V) 
= \mathcal E_m(\mathcal U)\cap \mathcal E_m(\mathcal V).
\]
In particular,
$\widehat{\mathcal S}_{1,m}
= \operatorname{row}(\mathbf X)\cap \operatorname{row}(\mathbf Y)
= \mathcal E_m(\mathcal G\cap K\mathcal G)$,
and inductively the sampled chain $\{\widehat{\mathcal S}_{j,m}\}$ mirrors the
operator-theoretic chain $\{\mathcal S_j\}$.
\end{proposition}

\begin{proof}
The inclusion $\mathcal E_m(\mathcal U\cap \mathcal V) 
\subseteq \mathcal E_m(\mathcal U)\cap \mathcal E_m(\mathcal V)$ is immediate.
For the reverse, let
$\mathbf y\in \mathcal E_m(\mathcal U)\cap \mathcal E_m(\mathcal V)$.
Then $\mathbf y=\mathcal E_m(u)=\mathcal E_m(v)$ with
$u\in\mathcal U$, $v\in\mathcal V$,
so $\mathcal E_m(u-v)=0$ and $u-v\in \mathcal G+K\mathcal G$.
By Assumption~\ref{ass:sampling-identifiability}, $u=v\in \mathcal U\cap\mathcal V$, 
hence $\mathbf y\in \mathcal E_m(\mathcal U\cap \mathcal V)$.
\end{proof}

The following theorem establishes convergence at the first nontrivial 
intersection ($j=1$).

\begin{theorem}[Convergence of the empirical Filtered EDMD operator]
\label{thm:convergence}
At the first filtered level, set
\begin{equation}\label{eq:W-S1m-def}
W_{\mathcal S_1,m}
:=\bigl\{\mathbf c\in\mathbb C^n:
\mathbf c^*\mathbf Y\in\widehat{\mathcal S}_{1,m}\bigr\},
\qquad
\boldsymbol\Pi_{\mathcal S_1,m}^{\mathrm c}
:=\operatorname{proj}_{W_{\mathcal S_1,m}},
\end{equation}
where $\operatorname{proj}$ denotes Euclidean-orthogonal projection in
$\mathbb C^n$, and define
\begin{equation}\label{eq:AS-empirical}
\mathbf A_{\mathcal S_1,m}^{\mathrm c}
:=\bigl(\boldsymbol\Pi_{\mathcal S_1,m}^{\mathrm c}\mathbf Y\bigr)
\mathbf X^\dagger.
\end{equation}
Let $A_{\mathcal S_1,m}^{\mathrm c}$ be the induced operator
from~\eqref{eq:empirical-induced-operator}. For i.i.d.\ samples
$\mathbf x_i\sim\mu$, suppose that $\{g_j\}_{j=1}^n$ is
$L^2(\mu)$-linearly independent, namely
\[
\|\langle\mathbf g,\mathbf c\rangle\|_\mu>0
\qquad\text{for every }\mathbf c\ne\mathbf0.
\]
Then, almost surely,
\[
\bigl\|A_{\mathcal S_1,m}^{\mathrm c}
-A_{\mathcal S_1}^{\mathrm c}\bigr\|_{\mathcal L(\mathcal G)}
\longrightarrow0,
\qquad
A_{\mathcal S_1}^{\mathrm c}
=P_{\mathcal S_1}^{\mathrm c}\circ(K|_{\mathcal G}),
\]
where the operator norm is induced by the $L^2(\mu)$ norm on the
finite-dimensional space $\mathcal G$.
\end{theorem}

\begin{proof}
Let
\[
\mathbf G_{ij}:=\langle g_i,g_j\rangle_\mu,
\qquad
\mathbf B_{ij}:=\langle Kg_i,g_j\rangle_\mu.
\]
The strong law of large numbers and the assumed linear independence give,
almost surely,
\[
\frac1m\mathbf X\mathbf X^*\longrightarrow\mathbf G\succ0,
\qquad
\frac1m\mathbf Y\mathbf X^*\longrightarrow\mathbf B.
\]
Thus $\mathbf X$ has full row rank for all sufficiently large $m$, and
\[
\mathbf A_{\mathcal S_1,m}^{\mathrm c}
=\boldsymbol\Pi_{\mathcal S_1,m}^{\mathrm c}
 \Bigl(\frac1m\mathbf Y\mathbf X^*\Bigr)
 \Bigl(\frac1m\mathbf X\mathbf X^*\Bigr)^{-1}.
\]
By Proposition~\ref{prop:sample-projector-convergence},
$\boldsymbol\Pi_{\mathcal S_1,m}^{\mathrm c}$ converges almost surely to
$\boldsymbol\Pi_{\mathcal S_1}^{\mathrm c}$, the Euclidean projector onto
$W_{\mathcal S_1}=\Psi_K^{-1}(\mathcal S_1)$. Consequently,
\[
\mathbf A_{\mathcal S_1,m}^{\mathrm c}
\longrightarrow
\boldsymbol\Pi_{\mathcal S_1}^{\mathrm c}\mathbf B\mathbf G^{-1}.
\]
The limiting matrix represents $P_{\mathcal S_1}^{\mathrm c}K|_{\mathcal G}$
in the dictionary $\mathbf g$: applying $P_{\mathcal S_1}^{\mathrm c}$
componentwise gives
$P_{\mathcal S_1}^{\mathrm c}(K\mathbf g)
=\boldsymbol\Pi_{\mathcal S_1}^{\mathrm c}K\mathbf g$, whose entries lie in
$\mathcal S_1\subseteq\mathcal G$, and its coefficient matrix in
$\mathbf g$ is
$\boldsymbol\Pi_{\mathcal S_1}^{\mathrm c}\mathbf B\mathbf G^{-1}$.
Matrix convergence and operator convergence are equivalent on the
finite-dimensional space $\mathcal G$, which proves the claim.
\end{proof}

Theorem~\ref{thm:convergence} treats $j=1$; the same population--sample
correspondence extends along the chain as follows.

\begin{remark}[Convergence along the full chain]\label{rmk:full-chain-convergence}
For a nested i.i.d.\ sample sequence, suppose that
Assumption~\ref{ass:sampling-identifiability} holds almost surely for all
sufficiently large $m$. Then
for each $j \geq 1$ the empirical coordinate-orthogonal Filtered EDMD operator
$A_{\mathcal S_j,m}^{\mathrm c}$ converges almost surely to
$A_{\mathcal S_j}^{\mathrm c}$ as $m\to\infty$.
Indeed, Proposition~\ref{prop:Em-intersections} gives
$\widehat{\mathcal S}_{j,m}=\mathcal E_m(\mathcal S_j)$, and injectivity of
$\mathcal E_m$ on $\mathcal G+K\mathcal G$ then gives
\[
W_{\mathcal S_j,m}=\Psi_K^{-1}(\mathcal S_j)=W_{\mathcal S_j}.
\]
Thus the empirical and population coefficient projectors agree eventually,
and the matrix argument in Theorem~\ref{thm:convergence} applies at every
fixed filtered level, including the terminal space
$\mathcal S_\ast=\mathcal V_{\max}^{\mathcal G}$.
The unfiltered level is ordinary EDMD and is treated separately:
$A_{E,m}\to A_E$. Collecting these, the consistent empirical family is
\[
\{A_{E,m}\}\cup
\{A_{\mathcal S_j,m}^{\mathrm c}:j\geq1\},
\]
from unfiltered EDMD through the intermediate coordinate-orthogonal levels to
the maximal-invariant restriction.
\end{remark}

\section{Numerical experiments}\label{sec:experiments}

We evaluate Filtered EDMD on three dynamical systems, chosen to 
illustrate the main challenges practitioners face when selecting 
observable dictionaries heuristically.%
\footnote{We have chosen low-dimensional examples so that eigenvalues, 
eigenfunctions, and matrix structures can be inspected directly.  
The method itself has no inherent dimensional restriction: 
Algorithm~\ref{alg:iter-intersect} operates on the $n\times m$ 
data matrices and scales as described in Remark~\ref{rmk:complexity}.}

\begin{table}[h]
\centering
\begin{tabular}{lll}
\toprule
System & Dictionary type & Challenge \\
\midrule
Kronecker flow on $\mathbb T^2$ & Eigenfunctions $+$ redundant terms & Redundant observables \\
Polynomial nonlinear system & Truncated polynomial basis & Insufficient observables \\
Van der Pol oscillator & Polynomial--trigonometric basis & Uninformative observables \\
\bottomrule
\end{tabular}
\caption{Test systems and the dictionary challenges they illustrate.}
\label{tab:test-systems}
\end{table}

The examples are ordered by increasing difficulty, matching the two structural
regimes of our theory. In the first two---the Kronecker flow
(Section~\ref{sec:exp-kronecker}) and the polynomial system
(Section~\ref{sec:exp-polynomial})---the dictionary $\mathcal G$ is not
Koopman-invariant, yet the one-step intersection $\mathcal G\cap K\mathcal G$ is
invariant and equals the maximal invariant subspace
$\mathcal V_{\max}^{\mathcal G}$: the Kronecker dictionary contains all relevant
eigenfunctions and the task is to separate them from redundant terms, while the
polynomial dictionary captures only some and the task is stability under limited
data. By Proposition~\ref{prop:canonical-projectors-edmd} and
Remark~\ref{Relation_to_SSD}, coordinate- and $L^2(\mu)$-Filtered EDMD, EDMD on
the reduced dictionary, and SSD/T-SSD then coincide at the population level on
this block ($=K|_{\mathcal S}$); all methods retain the genuine eigenvalues, so
the comparison isolates \emph{finite-sample} behavior---matrix structure and
stability under resampling---where they differ only through their
subspace-selection criteria (intersection, invariance, forward--backward
consistency, or residual thresholding) and their sensitivity to sampling error
and tolerance. In the third and most demanding example---the Van der Pol
oscillator (Section~\ref{sec:exp-vdp})---the dictionary contains no nontrivial
Koopman eigenfunctions and the maximal invariant subspace is trivial, so the
invariance- and residual-based methods have nothing to recover; yet the
coordinate-projected Filtered EDMD identifies candidates aligned with the
unstable-equilibrium Koopman lattice, whereas the other algorithms do not
reveal a clear lattice structure.

The benchmarks include EDMD, T-SSD, RFB-EDMD, ResDMD, and Filtered EDMD
(F-EDMD in tables). Exact SSD is not plotted
separately; its population relation to the terminal invariant space is given in
Remark~\ref{Relation_to_SSD}. For matrix heatmaps, reduced T-SSD and RFB-EDMD
operators are lifted to the original dictionary coordinates, while the ResDMD
panel is reconstructed from the retained EDMD eigenpairs.
\subsection{Kronecker flow on the torus: redundant observables}
\label{sec:exp-kronecker}

{
This example tests \emph{operator matrix structure} when the dictionary 
contains all relevant Koopman eigenfunctions alongside redundant, 
non-invariant terms.  
Since every method retains the $10$ genuine eigenvalues in this setting, the
comparison turns on whether spurious content remains and how faithfully the
resulting matrix reproduces the Koopman block structure.
}

\paragraph{Setup.}
Consider the Kronecker flow on $\mathbb{T}^2$ defined by 
$(\dot\theta_1,\dot\theta_2)=(1,\alpha)$, embedded in $\mathbb C$
via $z(\theta)=e^{i\theta_1}+e^{i\theta_2}$.  
When $\alpha\in\mathbb{Q}$ the trajectory is periodic; when 
$\alpha\notin\mathbb{Q}$ it fills the torus densely 
(Fig.~\ref{fig:kronecker-flow}).  
In the experiments below, we fix $\alpha=\sqrt{2}$ and generate discrete-time data by sampling the continuous-time flow with step size $\Delta t=0.2$.
The Koopman eigenfunctions are Fourier modes 
$\psi_{k,\ell}(\theta)=e^{i(k\theta_1+\ell\theta_2)}$ with 
discrete-time eigenvalues $\mu_{k,\ell}=e^{i(k+\ell\alpha)\Delta t}$.  
We construct a dictionary $\mathbf{g}$ of $n=18$ observables: 
$10$ Koopman eigenfunctions $\{\psi_j\}_{j=1}^{10}$ and 
$8$ additional non-invariant terms $\{r_k\}_{k=1}^{8}$.  
The maximal Koopman-invariant subspace in 
$\operatorname{span}(\mathbf{g})$ is the $10$-dimensional 
eigenfunction span $\mathcal{V}_{\max}^{\mathcal{G}}$.
By the standard Fourier-basis representation of the Koopman operator for rotations,
as illustrated for circle rotations in~\cite{Mezic2022Numerical}, the restriction of $K$ to $\mathcal{V}_{\max}^{\mathcal{G}}$ is exactly represented by a diagonal matrix whose diagonal entries are the selected Koopman eigenvalues.
We collect $m=50$ snapshot pairs with initial conditions drawn 
uniformly on $\mathbb{T}^2$.

\begin{figure}[!htbp]
    \centering
    \includegraphics[width=0.95\textwidth]{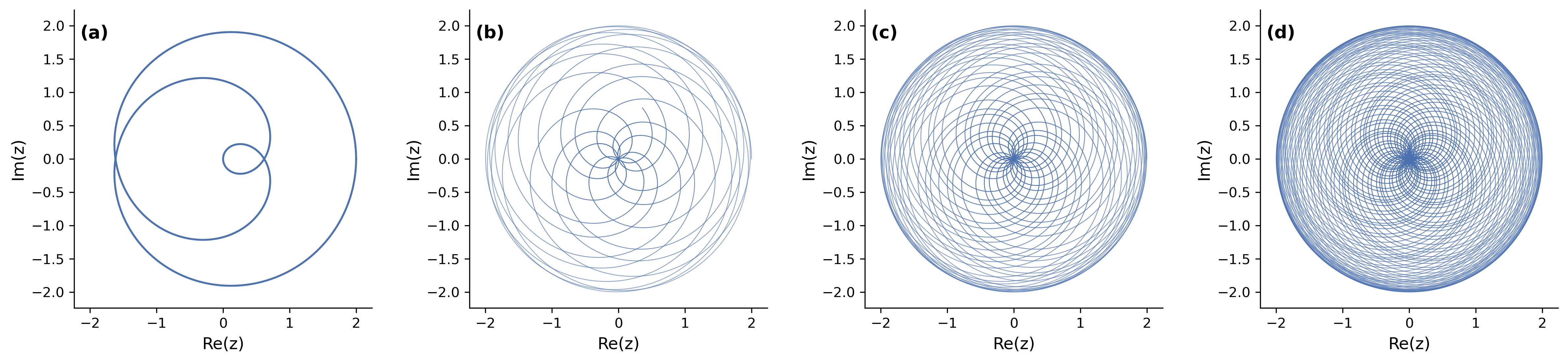}
    \caption{Kronecker flow trajectories in $\mathbb C$ via
    $z(\theta)=e^{i\theta}+e^{i\alpha\theta}$: 
    (a)~rational $\alpha$ (periodic); 
    (b)--(d)~irrational $\alpha$ with increasing time 
    (trajectory fills the curve densely).}
    \label{fig:kronecker-flow}
\end{figure}

{
\paragraph{Spectral recovery and matrix structure.}
The EDMD matrix is $18\times 18$ with full rank; its spectrum 
contains the $10$ correct Koopman eigenvalues plus $8$ spurious ones 
(Fig.~\ref{fig:torus-eigs}).  
Filtered EDMD produces a rank-$10$ matrix recovering the $10$ eigenvalues to
numerical precision; the forward--intersection chain terminates after
one step since $\mathcal{G}\cap K\mathcal{G}$ already coincides 
with $\mathcal{V}_{\max}^{\mathcal{G}}$.  
T-SSD, RFB-EDMD, and ResDMD also retain the $10$ genuine eigenvalues but differ
in their reconstructed matrix structure.
}

\begin{figure}[!htbp]
    \centering
    \includegraphics[width=0.9\textwidth]{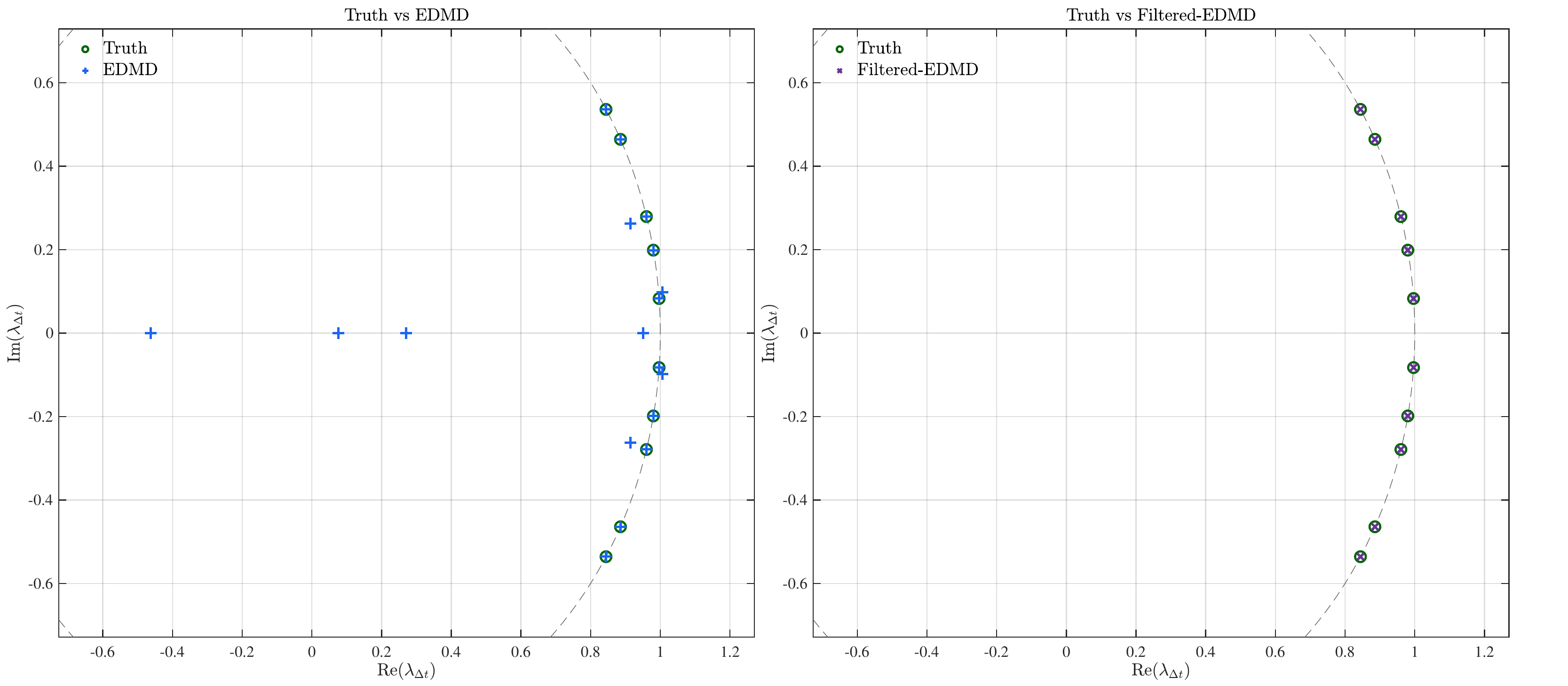}
    \caption{{Spectra of EDMD (left) and Filtered EDMD (right) on 
    the Kronecker flow.  
    EDMD produces $8$ spurious eigenvalues (red); 
    Filtered EDMD recovers the $10$ correct eigenvalues to numerical precision.}}
    \label{fig:torus-eigs}
\end{figure}

{
Figure~\ref{fig:torus-heatmaps} shows heatmaps of the absolute values
of the matrix representations described above.
All five data-driven representations capture the dominant $10\times 10$ diagonal block
corresponding to the {$10$-dimensional 
eigenfunction span $\mathcal{V}_{\max}^{\mathcal{G}}$.}
However, EDMD, ResDMD, and T-SSD exhibit substantial leakage into 
the complementary directions, whereas Filtered EDMD and RFB-EDMD 
produce clean block-diagonal matrices matching the theoretical 
Koopman representation.
}

\begin{figure}[!htbp]
    \centering
    \includegraphics[width=0.9\textwidth]{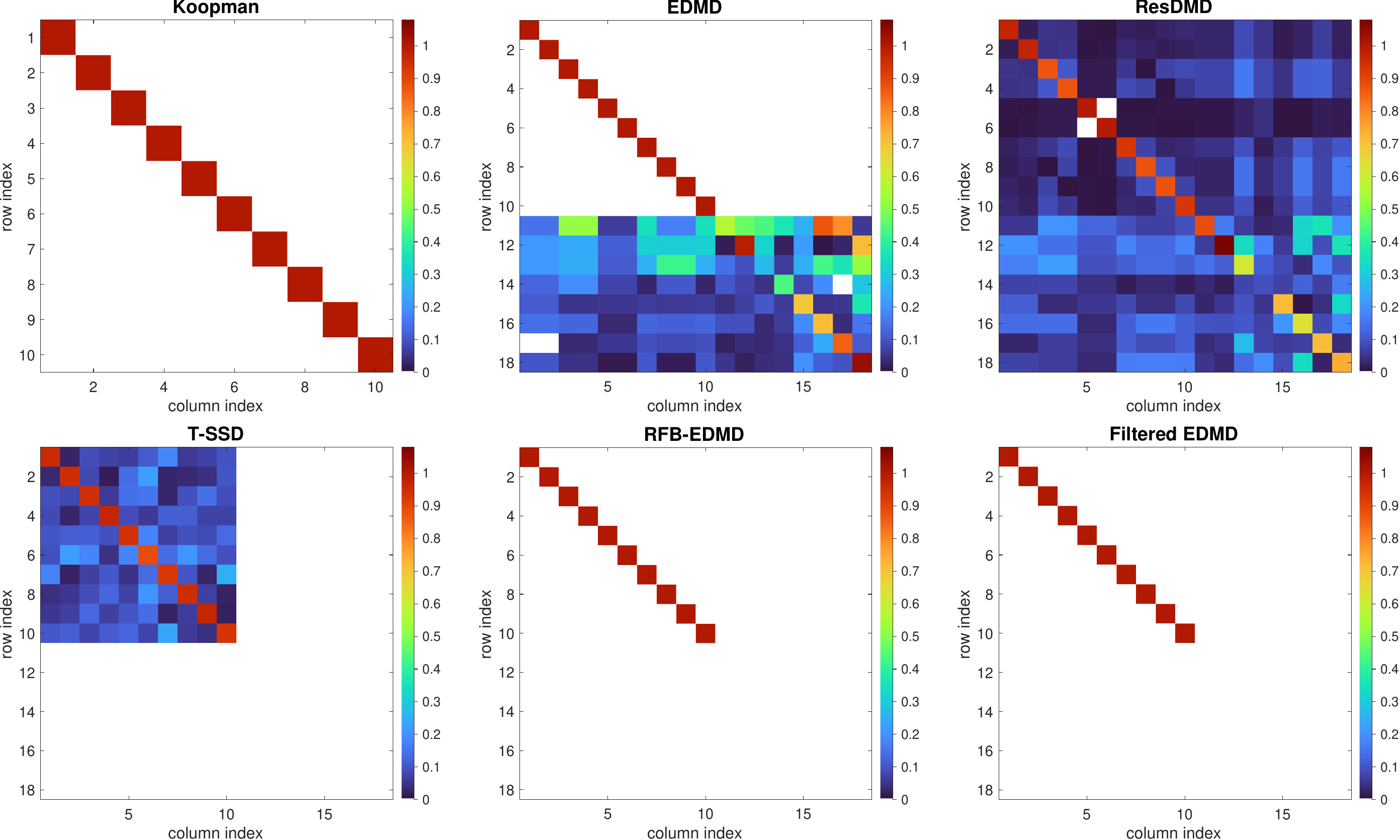}
    \caption{{Matrix heatmaps (absolute value) for the reference Koopman matrix {with respect to \(\{\psi_j\}_{j=1}^{10}\)}
    and the five data-driven methods on the Kronecker flow ($n=18$, $m=50$).
    Among the data-driven methods, only Filtered EDMD and RFB-EDMD reproduce the clean
    block-diagonal structure of the Koopman operator.}}
    \label{fig:torus-heatmaps}
\end{figure}

{
\paragraph{Tolerance sensitivity.}
Figure~\ref{fig:torus-rank1} shows the reported matrix rank as a function
of the tolerance parameter $\varepsilon$ 
(cf.\ Remark~\ref{rmk:numerical-rank}) for the four methods that 
admit one.  
Over the tested tolerance ranges, every method attains rank $10$.
T-SSD reaches rank $10$ at a larger tolerance than RFB-EDMD, consistent with
\citep{RFB-EDMD}.
Filtered EDMD exhibits a sharp transition, stabilizing at rank $10$ 
once $\varepsilon$ passes below a critical threshold.  
Figure~\ref{fig:torus-rank2} illustrates three intermediate stages 
(ranks $13$, $11$, $10$), showing the progressive removal of 
spurious off-diagonal couplings.
}

\begin{figure}[!htbp]
    \centering
    \includegraphics[width=0.9\textwidth]{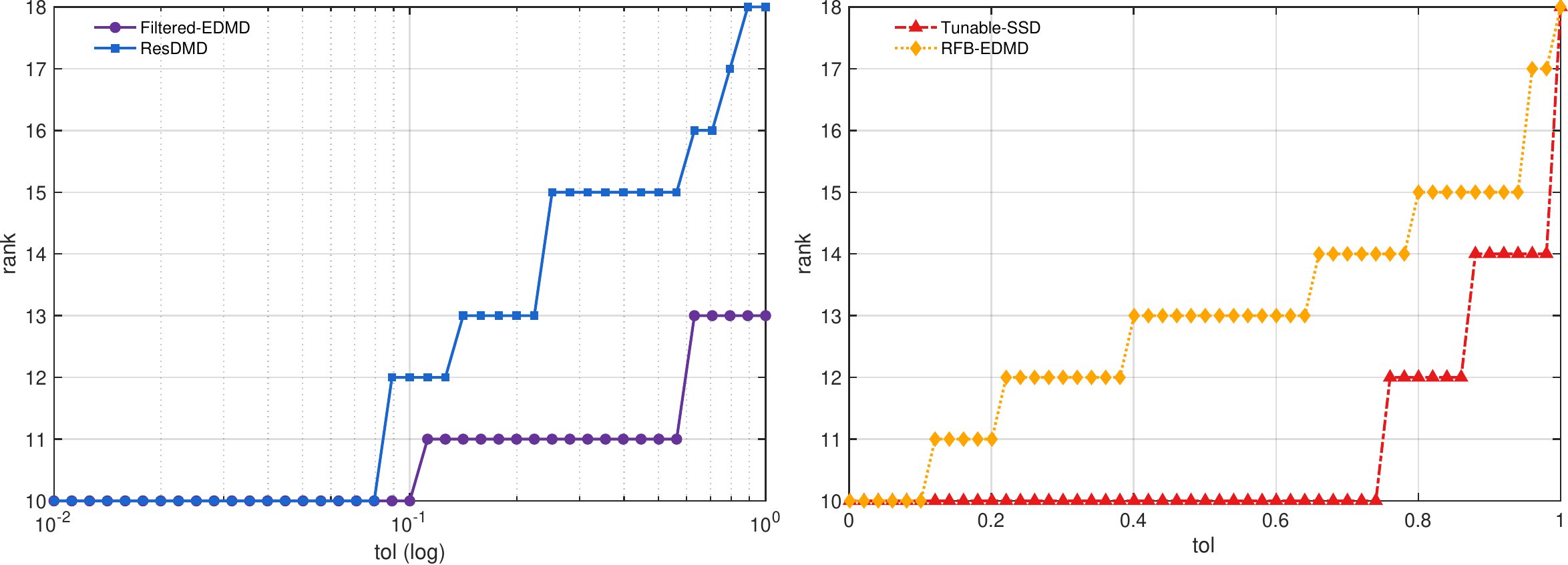}
    \caption{{Reported matrix rank versus tolerance $\varepsilon$ for the
    four parameterized methods (Kronecker flow).}}
    \label{fig:torus-rank1}
\end{figure}

\begin{figure}[!htbp]
    \centering
    \includegraphics[width=0.9\textwidth]{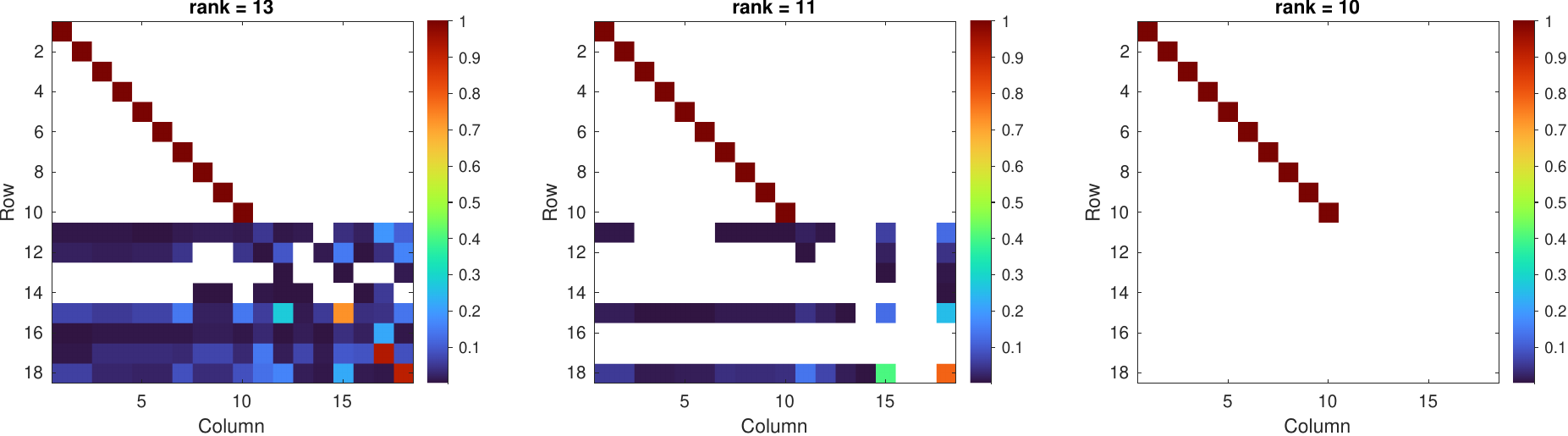}
    \caption{Filtered EDMD matrices at three tolerance levels, 
    showing progressive removal of spurious couplings 
    until the clean rank-$10$ structure is recovered.}
    \label{fig:torus-rank2}
\end{figure}

\subsection{Polynomial nonlinear system: insufficient observables}
\label{sec:exp-polynomial}

This example tests \emph{spectral recovery and algorithmic stability} 
when the dictionary captures only a subset of the Koopman 
eigenfunctions---a common situation when the dictionary is chosen 
by truncating a general basis.

\paragraph{Setup.}
We consider a 4-dimensional polynomial ODE admitting an exact 
finite-dimensional Koopman representation \parencite{Koopman_Form}:
\[
\begin{aligned}
\dot x_1 &= a_1x_1,  &
\dot x_2 &= a_2x_2+\alpha_1 x_1^3,  \\
\dot x_3 &= a_3x_3+\alpha_2 x_1x_2+\alpha_3 x_2^2,  &
\dot x_4 &= a_4x_4+\alpha_4 x_1x_2x_3.
\end{aligned}
\]
Throughout this subsection, we fix
$(a_1,a_2,a_3,a_4)=0.1(\log 17,\log 15,\log 3,i\log 2)$
and $\alpha_j=-0.2, \forall j$.

A minimal faithful linear representation uses $19$ monomial 
observables and an upper-triangular matrix whose 
diagonal entries are the Koopman eigenvalues $\{\lambda_i\}_{i=1}^{19}$ 
(see \citep{Koopman_Form} for details).  
Some dictionary elements (e.g., $x_1$, $x_1^3$) are already eigenfunctions; 
others (e.g., $x_2$, $x_3$) become eigenfunctions only when combined 
with upstream monomials that cancel coupling terms.

We deliberately restrict to the first $15$ of the $19$ observables, so 
that $\mathcal{G}=\operatorname{span}\{g_1,\dots,g_{15}\}$ is not 
Koopman-invariant.  
The truncated span still contains $9$ eigenfunctions (those whose 
construction requires only monomials within the first $15$), but the 
eigenfunction associated with $\lambda_4$ lies outside $\mathcal{G}$ 
since it requires higher-order monomials.  
The resulting representation is therefore \emph{semiconjugate} to the 
full faithful representation: it captures part of the spectrum and can 
reconstruct $x_1,x_2,x_3$ but not $x_4$.  
We generate $m=25$ snapshot pairs from initial states sampled uniformly 
in $[-1,1]^4$ with discretization step $\Delta t=0.1$.

{
\paragraph{Spectral recovery.}
The EDMD spectrum (Fig.~\ref{fig:poly-eigs}) recovers the $9$ true 
eigenvalues but is contaminated by $6$ spurious ones, including two 
that lie deceptively close to the exact spectrum.  
Filtered EDMD eliminates all detected spurious eigenvalues and matches the
$9$ Koopman eigenvalues to numerical precision; the chain terminates after one step.
With the selected parameters, T-SSD and ResDMD also retain the $9$ eigenvalues
on this dataset. RFB-EDMD does not remove all spurious eigenvalues at $m=25$,
but succeeds once the sample size reaches $m=100$.
}

\begin{figure}[!htbp]
    \centering
    \includegraphics[width=0.9\textwidth]{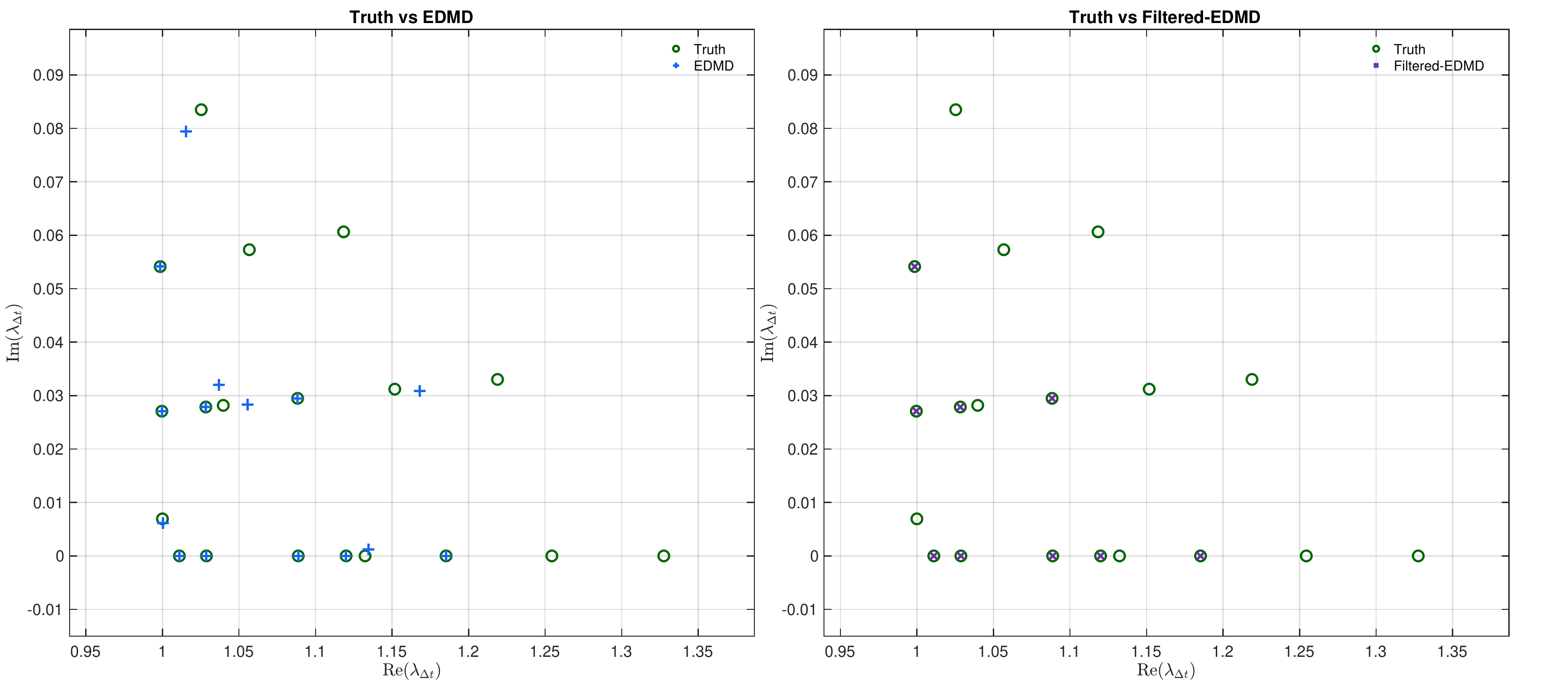}
    \caption{{Spectra of EDMD and Filtered EDMD on the polynomial system 
    ($n=15$, $m=25$).  
    EDMD retains $6$ spurious eigenvalues (red); 
    Filtered EDMD recovers the $9$ true eigenvalues to numerical precision.}}
    \label{fig:poly-eigs}
\end{figure}

{
Figure~\ref{fig:poly-heatmaps} visualizes the matrix representations: the Koopman matrix is $19\times 19$ (full dictionary), and all computed matrices are $15\times 15$.  
Filtered EDMD recovers, to numerical precision, the block corresponding to the $9$
eigenfunctions in $\mathcal{G}$---the upper-left $10\times 10$ block 
with the fourth row and column removed---while T-SSD displays a similar 
pattern with noticeable numerical residues.
}

\begin{figure}[!htbp]
    \centering
    \includegraphics[width=0.9\textwidth]{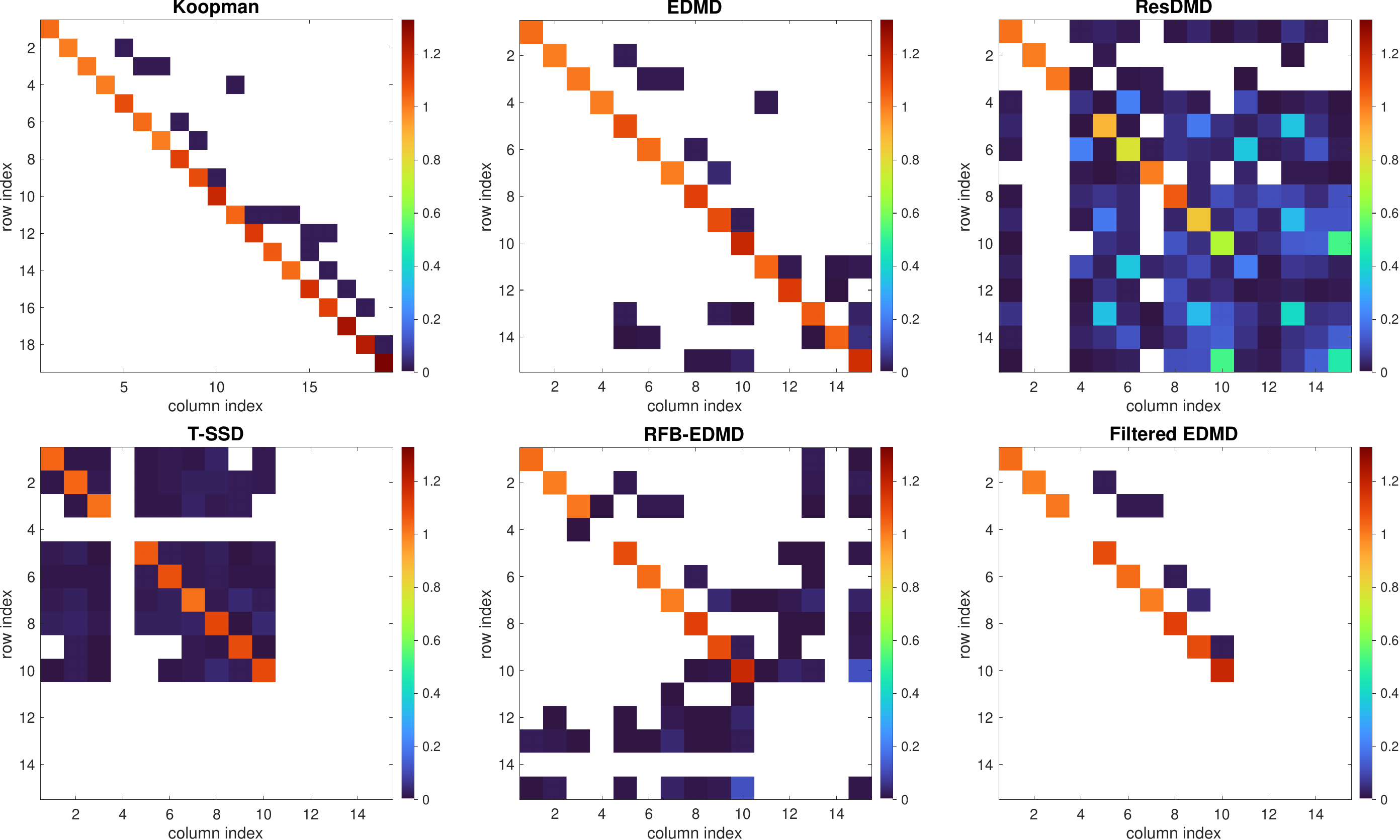}
    \caption{{Matrix heatmaps (absolute value) for the
    polynomial system.  The reference Koopman matrix is $19\times 19$; 
    all others are $15\times 15$.}}
    \label{fig:poly-heatmaps}
\end{figure}

\subsubsection{Stability under resampling}

To assess robustness beyond a single dataset, we perform a Monte Carlo 
experiment with $q=1000$ independent trials.  
Each trial generates fresh data matrices 
$(\mathbf{X},\mathbf{Y})\in\mathbb C^{15\times 25}$ from newly
sampled initial conditions.  
Tolerance parameters are pre-tuned once on a separate dataset and 
held fixed across all trials, so the comparison reflects intrinsic 
algorithmic stability rather than per-trial re-optimization.

We measure spectral error by the average modulus of spurious eigenvalues,
\[
d_{\mathrm{spec}}(\widehat{\Lambda},\Lambda_{\mathrm{gt}})
:=
\begin{cases}
\frac{1}{|\mathcal E|}\sum_{\lambda\in\mathcal E}|\lambda|, & \mathcal E\neq\varnothing,\\
0, & \mathcal E=\varnothing,
\end{cases}
\quad
\mathcal E:=\Big\{\lambda\in\widehat{\Lambda}: 
\min_{\mu\in\Lambda_{\mathrm{gt}}}|\lambda-\mu|>10^{-10}\Big\},
\]
{
where, as in the implementation, $\widehat\Lambda$ contains only finite
computed eigenvalues with $|\lambda|>10^{-10}$. The same threshold serves as a stringent matching tolerance, safely above floating-point roundoff.
}
Trajectory error at step $k$ is the relative deviation
$e_{\mathrm{traj}}(k) := 
\| \widehat{\mathbf{x}}(k) - \mathbf{x}(k)\|_2 / 
\|\mathbf{x}(k)\|_2 \times 100\%$.

{
Table~\ref{tab:mc-spec} summarizes spectral error statistics.  
Filtered EDMD has no detected spurious eigenvalues in any of the $1000$ trials;
ResDMD and T-SSD have none in $897$ and $839$ trials, respectively, whereas
EDMD and RFB-EDMD retain at least one in every trial. This statistic measures
false-positive eigenvalues only and does not penalize missing genuine ones.
Figure~\ref{fig:traj-box} shows the mean trajectory errors.
EDMD performs best initially due to its least-squares formulation, 
but error accumulates in the fourth component (which it attempts to 
predict despite missing the required eigenfunction).  
Filtered EDMD remains the most stable over time.  
These errors reflect algorithmic stability under fixed parameters and 
repeated resampling, not optimized prediction performance.
}

\begin{figure}[!htbp]
\centering
\begin{minipage}[c]{0.45\textwidth} 
\centering
\setlength{\tabcolsep}{5.8pt}
\begin{tabular}{lccc}
\toprule
Algorithm & Mean & Median & {No-spurious} \\
\midrule
EDMD      & 1.099  & 1.099  & 0 \\
\textbf{F-EDMD}  & \textbf{0}       & \textbf{0}       & \textbf{1000} \\
T-SSD     & 0.175 & 0       & 839 \\
RFB-EDMD  & 1.061  & 1.060  & 0 \\
{ResDMD}    & 0.119 & 0       & 897 \\
\bottomrule
\end{tabular}
\captionof{table}{{Spectral error $d_{\mathrm{spec}}$ over $1000$ Monte Carlo
trials. ``No-spurious'' counts trials with no unmatched nonzero estimate; it
does not measure spectral recall.}}
\label{tab:mc-spec}
\end{minipage}%
\hfill
\begin{minipage}[c]{0.50\textwidth} 
\centering
\includegraphics[width=\linewidth]{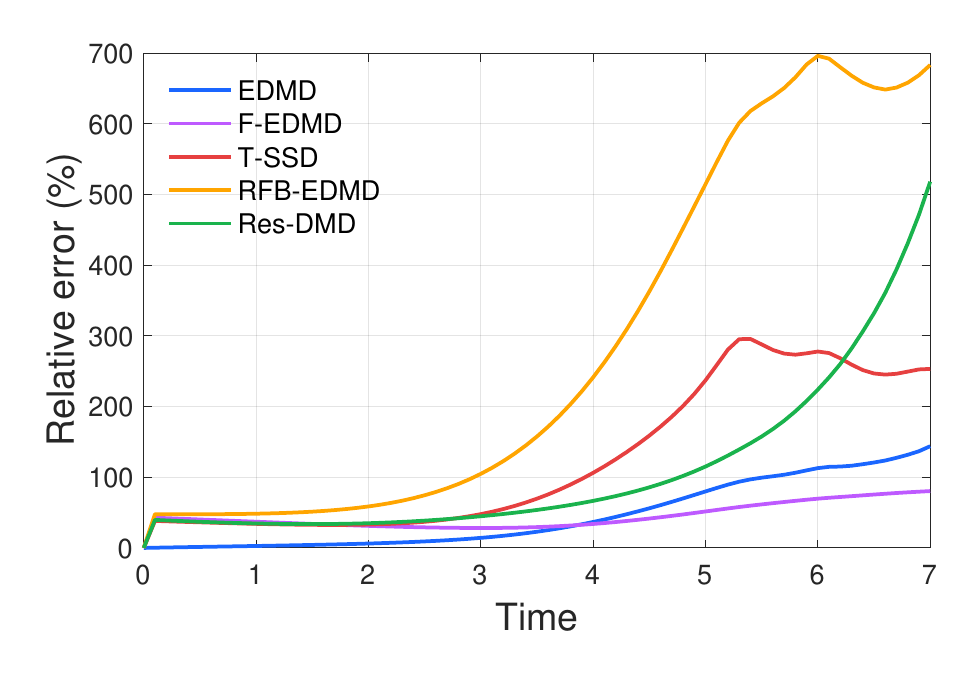}
\captionof{figure}{{Mean relative trajectory error (\%) over $1000$ trials
($n=15$, $m=25$).  
Filtered EDMD is the most stable; EDMD has low initial error 
but accumulates drift.}}
\label{fig:traj-box}
\end{minipage}
\end{figure}

\subsection{Van der Pol oscillator: spectrally uninformative observables}
\label{sec:exp-vdp}

{
This is the most demanding test: the dictionary contains 
\emph{no nontrivial Koopman eigenfunctions}, 
so its terminal invariant space is spectrally uninformative.
The question is whether Filtered EDMD---which requires only one-step 
compatibility ($\mathcal{G}\cap K\mathcal{G}$) rather than exact 
invariance---can still recover the underlying Koopman spectral structure.
}

\paragraph{Setup.}
We consider the Van der Pol oscillator in the weakly nonlinear regime,
\begin{equation*}
    \dot{x}_1 = x_2, \qquad
    \dot{x}_2 = \mu(1-x_1^2)x_2 - x_1,
\end{equation*}
with $\mu=0.3$ and state $\mathbf{x}=(x_1,x_2)^\top$.  

The system has a globally attracting limit cycle and a locally unstable equilibrium
(Fig.~\ref{fig:vdp-vectorfield}). 
These invariant objects induce different Koopman point spectra for the same
Van der Pol flow, depending on the chosen domain and observable space; see
Appendix~\ref{APP:VDP_Koop}.
The limit-cycle spectrum is generated by the phase eigenvalue \(i\omega\) and the
dominant transverse Floquet exponent \(\sigma<0\):
\begin{equation*}\label{eq:lattice-lc}
    \omega(\mu)=1-\frac{\mu^2}{16}+\mathcal{O}(\mu^4),
    \quad
    \sigma(\mu)=-\mu-\frac{\mu^3}{16}+\mathcal{O}(\mu^5),
    \quad
    \Lambda_{\rm lc}
    =
    \bigl\{
        ik\omega+\ell\sigma
        :
        k\in\mathbb Z,\ \ell\in\mathbb N
    \bigr\}.
\end{equation*}
The equilibrium spectrum is generated by the Jacobian eigenvalues at the unstable
equilibrium:
\begin{equation*}\label{eq:lattice-fp}
    \alpha_{\pm}
    =
    \frac{\mu\pm\sqrt{\mu^2-4}}{2},
    \qquad
    \Lambda_{\rm eq}
    =
    \bigl\{
        p\alpha_+ + q\alpha_-
        :
        p,q\in\mathbb N
    \bigr\}.
\end{equation*}
Both lattices are used as reference spectra in the numerical comparisons below.

The observable dictionary consists of all monomials $x_1^p x_2^q$ with
$p+q\leq 15$ and the trigonometric observables
\[
\cos(k\theta_0),\ \sin(k\theta_0),
\qquad k=1,\ldots,15,
\qquad \theta_0=\operatorname{atan2}(x_2,x_1),
\]
giving $n=166$ observables.
Apart from the constant function, the Koopman eigenfunctions of interest are
intricate functions of state, and none lies in $\operatorname{span}(\mathbf g)$.
We sample $m=3320$ initial conditions uniformly from $[-2,2]^2$ and 
advance each by $\Delta t=0.01$ to form snapshot pairs.

\begin{figure}[!htbp]
    \centering

    \begin{subfigure}[t]{0.31\textwidth}
        \centering
        \includegraphics[width=\textwidth]{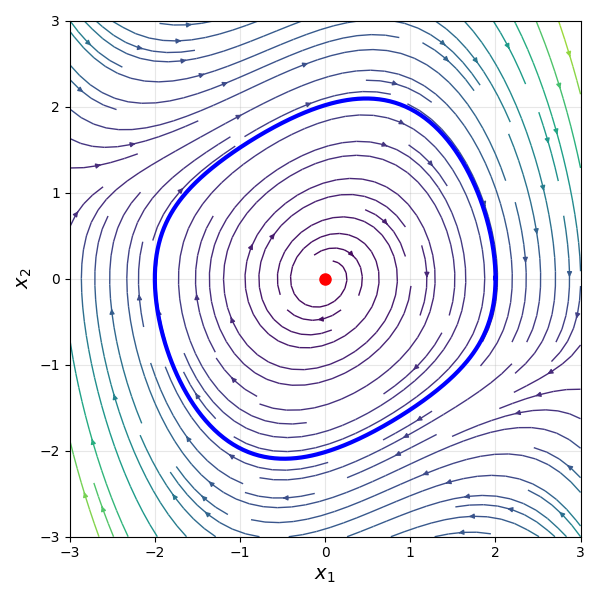}
        \caption{Vector field.}
        \label{fig:vdp-vectorfield}
    \end{subfigure}
    \hfill
    \begin{subfigure}[t]{0.31\textwidth}
        \centering
        \includegraphics[width=\textwidth]{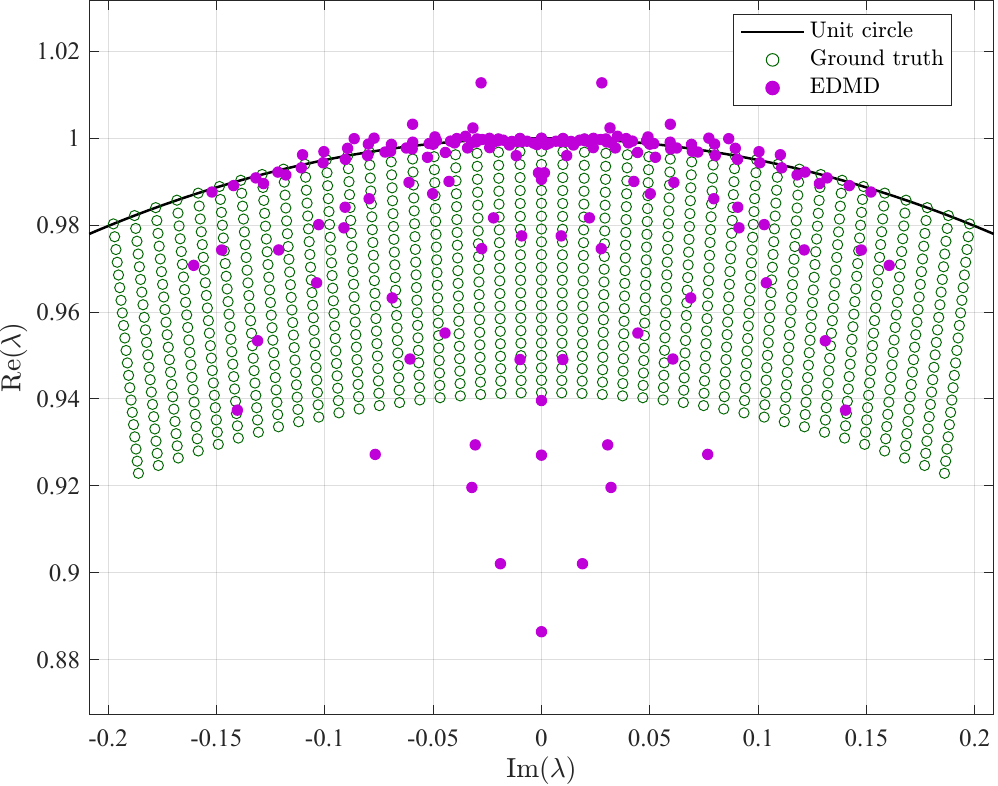}
        \caption{Discrete-time eigenvalues.}
        \label{fig:edmd-vdp-eigs-discrete}
    \end{subfigure}
    \hfill
    \begin{subfigure}[t]{0.31\textwidth}
        \centering
        \includegraphics[width=\textwidth]{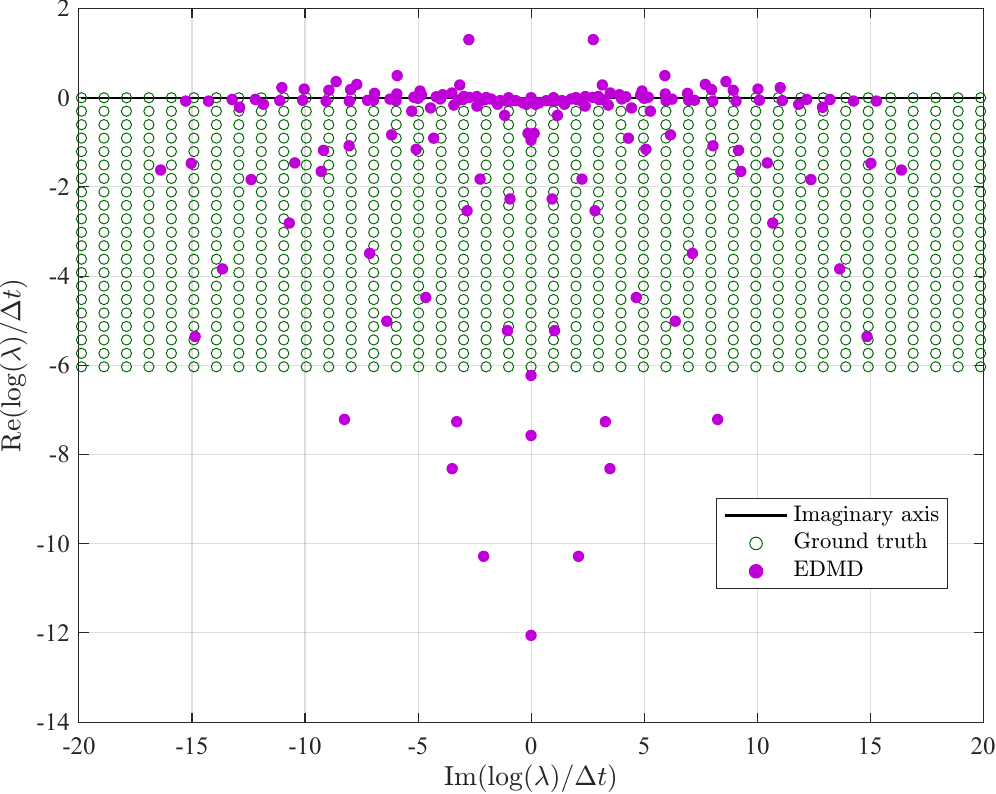}
        \caption{Continuous-time eigenvalues.}
        \label{fig:edmd-vdp-eigs-continuous}
    \end{subfigure}

    \caption{{Van der Pol oscillator ($\mu=0.3$). 
    Left: vector field with the limit cycle (blue), the unstable equilibrium 
    at the origin (red), and streamlines colored by flow speed. 
    Middle and right: EDMD eigenvalues (purple) versus the limit-cycle lattice
    $\Lambda_{\rm lc}$ (green circles) in discrete and continuous time, respectively.}}
\end{figure}

{
\paragraph{Spectral comparison across methods.}
EDMD does not resolve the limit-cycle lattice $\Lambda_{\rm lc}$: its $166$
eigenvalues bear little resemblance to that reference spectrum
(Figs.~\ref{fig:edmd-vdp-eigs-discrete} and~\ref{fig:edmd-vdp-eigs-continuous}).
For this diagnostic, we select the comparison methods' tolerances to
maximize lattice agreement over the tested values:
for T-SSD, the invariance tolerance; 
for RFB-EDMD, the forward--backward consistency threshold; 
and for ResDMD, the residual cutoff. Filtered EDMD uses the default numerical-rank
cutoff throughout.
Figure~\ref{fig:method-comparison} shows the outputs obtained under this
protocol:
\begin{itemize}
    \item \textbf{T-SSD} and \textbf{RFB-EDMD} retain subspaces of dimensions
    $24$ and $31$, respectively, and partially reveal $\Lambda_{\rm lc}$.
    {\item \textbf{ResDMD} retains $3$ EDMD eigenvalues, each matching a
corresponding element of $\Lambda_{\rm lc}$ to numerical precision at
the selected residual cutoff.}
    \item \textbf{Filtered EDMD} ($\mathbf A_{\mathcal S_1,m}^{\mathrm c}$,
    one-step filtering) retains $84$ directions and reveals a clear portion of
    the equilibrium lattice $\Lambda_{\rm eq}$.
\end{itemize}
Thus the panels expose different spectral content rather than a common
reduced operator. Exact SSD terminates at the uninformative invariant core;
T-SSD and RFB-EDMD retain partial limit-cycle structure, while the one-step
intersection retains polynomial directions aligned with the equilibrium
spectrum. The projection-geometry results in Section~\ref{sec:vdp-geometry} make
this distinction explicit: on the same $\mathcal S_1$, the measure-free coordinate
projector approximates $\Lambda_{\rm eq}$ while the $L^2(\mu)$ projector {tends to approximate}
$\Lambda_{\rm lc}$.
}

\begin{figure}[!htbp]
    \centering
    \includegraphics[width=0.70\textwidth]{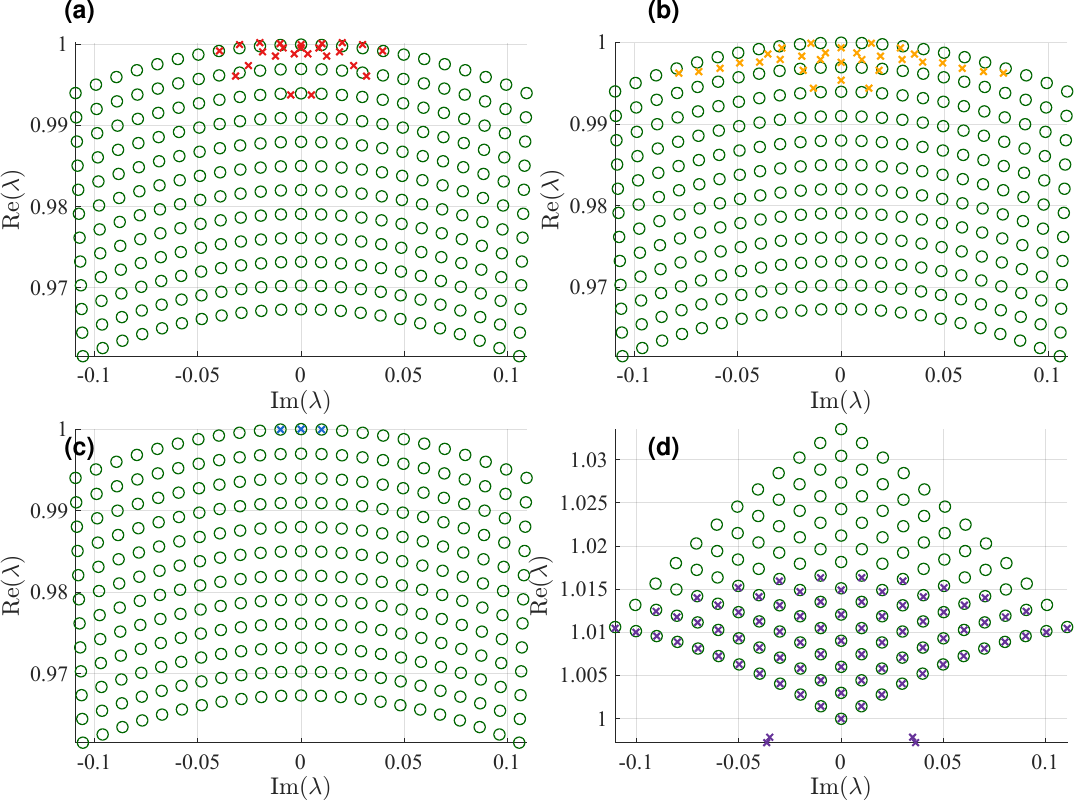}
    \caption{{Spectral outputs for the Van der Pol oscillator. Panels
    (a)--(c) compare T-SSD ($\dim=24$), RFB-EDMD ($\dim=31$), and ResDMD
    ($3$ eigenvalues) with $\Lambda_{\rm lc}$; panel (d) compares Filtered EDMD
    ($\mathcal S_1$, $\dim=84$) with $\Lambda_{\rm eq}$. Green circles denote
    the corresponding reference lattice.}}
    \label{fig:method-comparison}
\end{figure}

{
\paragraph{The filtration chain.}
For this dataset, the sampled construction produces four nontrivial levels
$\{\widehat{\mathcal S}_{j,m}\}_{j=1}^4$, yielding a
family of Filtered EDMD matrices 
$\{\mathbf A_{\mathcal S_j,m}^{\mathrm c}\}_{j=1}^4$.
As shown in Fig.~\ref{fig:chain-convergence}, the rank decreases 
along the chain while equilibrium-lattice structure remains visible.
}

\begin{figure}[!htbp]
    \centering
    \includegraphics[width=0.70\textwidth]{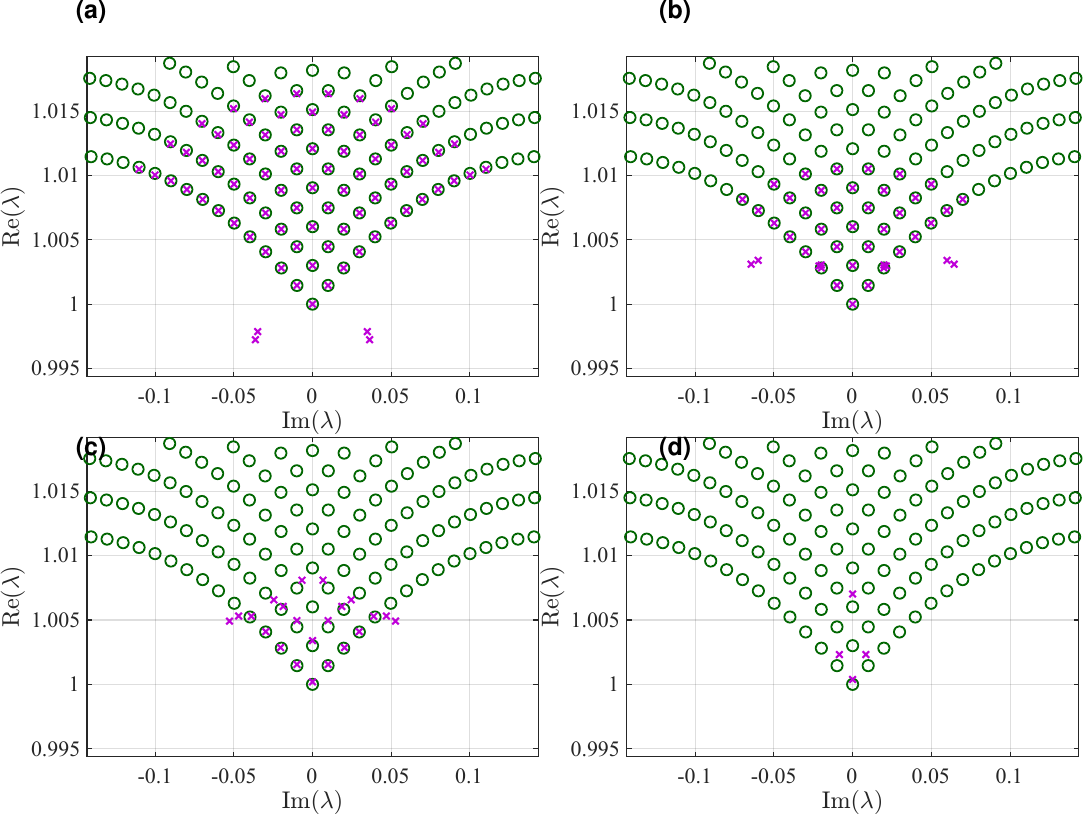}
    \caption{{Filtered EDMD eigenvalues along the forward--intersection 
    chain $\widehat{\mathcal S}_{j,m}$, $j=1,2,3,4$.
    Rank decreases monotonically while equilibrium-lattice structure remains
    visible.}}
    \label{fig:chain-convergence}
\end{figure}

{\subsubsection{Local eigenfunction approximation}\label{sec:vdp-eigfun-recovery}
Beyond eigenvalues, Filtered EDMD yields an approximation to the principal
local Koopman eigenfunction associated with the unstable equilibrium.
We extract this approximation from
\(\mathbf A_{\mathcal S_1,m}^{\mathrm c}\), selecting
the computed discrete-time eigenvalue \(\lambda_+=1.001452+0.009902i\),
which approximates the analytic \(e^{\alpha_+\Delta t}=1.0015+0.0099i\), and using the corresponding
left eigenvector to form $\widehat\varphi_{\lambda_+}$. Figure~\ref{fig:vdp-eigfun}(a)
shows its logarithmic magnitude, and Fig.~\ref{fig:vdp-eigfun}(b) displays the
associated phase- and amplitude-like level sets. Near the unstable equilibrium,
these structures resemble the local isochron--isostable geometry reported in
\citep{Mauroy2018}.}

\begin{figure}[!htbp]
    \centering
    \includegraphics[width=0.6\textwidth]{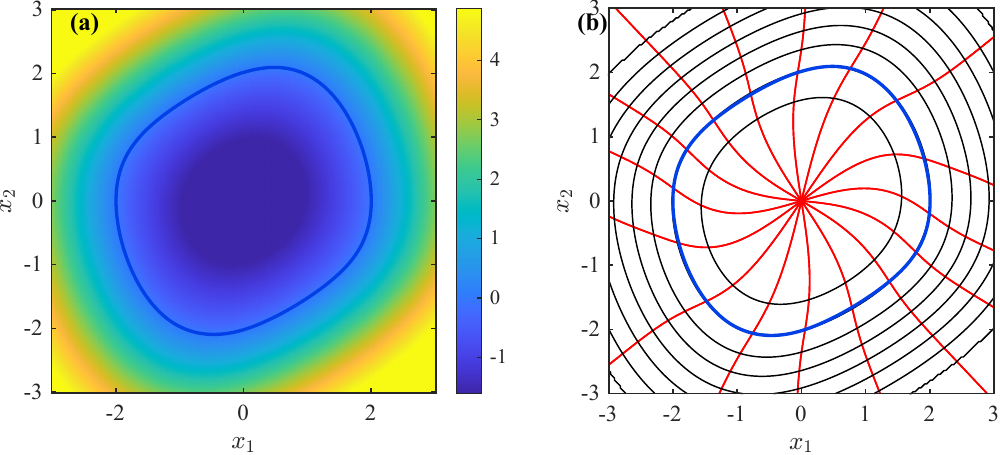}
    \caption{{Local Koopman-eigenfunction approximation from Filtered EDMD for
        the Van der Pol oscillator ($\mu=0.3$).  
        (a)~Log-magnitude of $\widehat\varphi_{\lambda_+}$.
        (b)~Phase-like (red, level sets of $\arg\widehat\varphi_{\lambda_+}$) and
        amplitude-like (black, level sets of $|\widehat\varphi_{\lambda_+}|$);
        limit cycle in blue.}}
    \label{fig:vdp-eigfun}
\end{figure}

\subsubsection{Projection geometry: coordinate versus \texorpdfstring{$L^2(\mu)$}{L2}, and sampling invariance}
\label{sec:vdp-geometry}

{We vary the projection \emph{geometry} and sampling \emph{measure} to compare
EDMD with the two canonical projectors from
Section~\ref{sec:canonical-projectors}. We use the polynomial observable
dictionary \(\mathcal G
=
\operatorname{span}
\left\{
x_1^p x_2^q:
p,q\geq 0,\;
p+q\leq 10
\right\}\)
and the first forward-intersection space
$\mathcal S_1=\mathcal G\cap K\mathcal G$. We compare the empirical matrices
$\mathbf A_{E,m}$, $\mathbf A_{\mathcal S_1,m}^{\mathrm c}$, and
$\mathbf A_{\mathcal S_1,m}^{\mu}$, whose respective population targets are
\[A_E=A_{\mathcal G}^{\mu}=P_{\mathcal G}^{\mu}\circ(K|_{\mathcal G}),\quad A_{\mathcal S_1}^{\mathrm c}=P_{\mathcal S_1}^{\mathrm c}\circ(K|_{\mathcal G}),\quad
A_{\mathcal S_1}^{\mu}=P_{\mathcal S_1}^{\mu}\circ(K|_{\mathcal G}).\]
The corresponding empirical matrices provide finite-sample approximations of
these population operators in the dictionary coordinates. At the population
level, $A_E$ and $A_{\mathcal S_1}^{\mu}$ depend on $\mu$ through their
$L^2(\mu)$-orthogonal projectors. The coordinate representation of
$K\mathbf g$ instead determines the geometry of
$A_{\mathcal S_1}^{\mathrm c}$ independently of $\mu$ for a given
$\mathbf{g}$ and $\mathcal S_1$.}

\paragraph{The $L^2(\mu)$-based spectra vary across sampling measures.}
{Figure~\ref{fig:vdp-robustness} shows that the finite-sample eigenvalues of
the EDMD matrix $\mathbf A_{E,m}$ and the $L^2(\mu)$ Filtered EDMD matrix
$\mathbf A_{\mathcal S_1,m}^{\mu}$ change across the three sampling
measures. For both matrices, the computed eigenvalues tend to lie near the
reference Koopman lattice associated with the region where the sampling
measure concentrates its mass. This observed sensitivity is consistent with
the $\mu$-dependence of their population targets $A_E$ and
$A_{\mathcal S_1}^{\mu}$.}

\begin{figure}[H]
    \centering
    \includegraphics[width=1\textwidth]
    {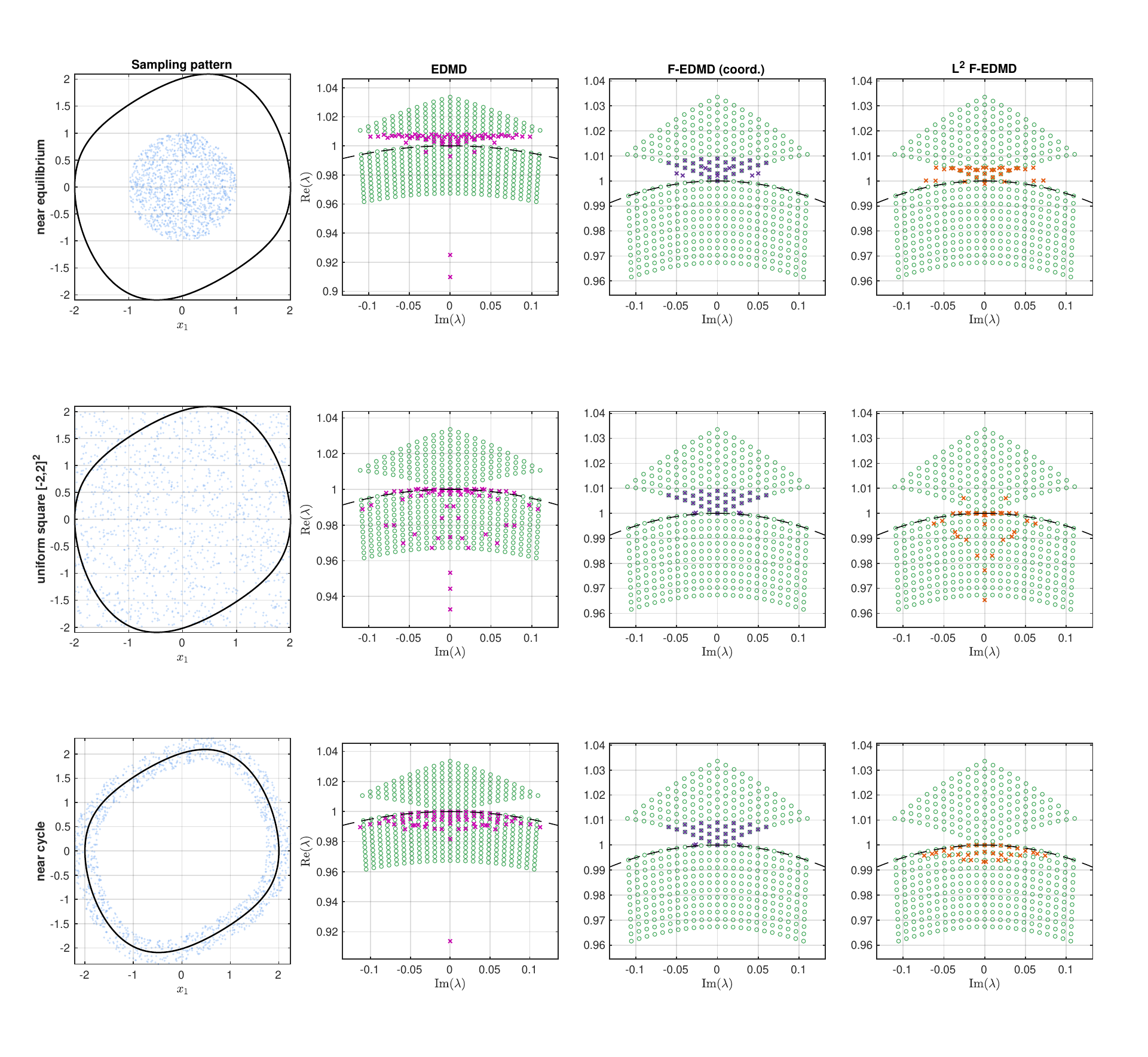}
    \caption{{
    Spectra under three sampling measures for the Van der Pol oscillator
    with a polynomial dictionary of total degree $10$ ($\Delta t=0.01$).
    Rows correspond to three sampling measures: a measure concentrated near the unstable equilibrium, the uniform measure on the square $[-2,2]^2$, and a measure concentrated near the limit cycle. Columns show the sampling
    pattern and the eigenvalues of EDMD, coordinate Filtered EDMD
    $\mathbf A_{\mathcal S_1,m}^{\mathrm c}$, and $L^2(\mu)$ Filtered EDMD
    $\mathbf A_{\mathcal S_1,m}^{\mu}$. Green circles mark the
    reference lattices $\Lambda_{\rm eq}$ and $\Lambda_{\rm lc}$. The EDMD and $L^2(\mu)$
    Filtered EDMD spectra change across sampling measures, whereas the
    coordinate Filtered EDMD spectrum remains visually aligned with
    $\Lambda_{\rm eq}$. }}
    \label{fig:vdp-robustness}
\end{figure}

\paragraph{Sampling-measure invariance of the coordinate F-EDMD target
$A_{\mathcal S_1}^{\mathrm c}$.}
{
Across the three sampling measures in Fig.~\ref{fig:vdp-robustness}, the
computed eigenvalues of $\mathbf A_{\mathcal S_1,m}^{\mathrm c}$ tend to
remain aligned with the equilibrium reference lattice $\Lambda_{\rm eq}$,
although their exact locations vary. This common tendency is consistent with
the fact that each empirical matrix approximates the same population operator \(A_{\mathcal S_1}^{\mathrm c}=P_{\mathcal S_1}^{\mathrm c}\circ(K|_{\mathcal G}).\)
For fixed $\mathcal G$ and $\mathcal S_1$, the coordinate representation of
$K\mathbf g$ determines the projection geometry of this operator independently
of $\mu$. Thus, sampling-measure invariance pertains to the population target
$A_{\mathcal S_1}^{\mathrm c}$, whereas its finite-sample approximation
$\mathbf A_{\mathcal S_1,m}^{\mathrm c}$ and the corresponding computed
eigenvalues can vary across samples.}

\paragraph{No fabricated spectrum.}
{
Finally, Filtered EDMD does not hallucinate structure when none is compatible: on the
pure angular dictionary $\{\cos(k\theta_0),\sin(k\theta_0)\}_{k=1}^{15}$ the sampled
one-step intersection is empty and $\mathbf A_{\mathcal S_1,m}^{\mathrm c}=0$, whereas
EDMD reports $30$ unit-circle eigenvalues that are not global Koopman eigenpairs.
}
\subsection{Computational cost}\label{sec:exp-cost}

Table~\ref{tab:cost} reports wall-clock times for all methods on the 
three test problems.%
\footnote{{All experiments were performed in MATLAB R2023b on a single machine with
18~GB of RAM.}
Timings are indicative; ranges report observed run-to-run variation.}

\begin{table}[H]
\centering
\begin{tabular}{l ccc}
\toprule
 & Kronecker & Polynomial & Van der Pol \\
 & ($n\!=\!18,\, m\!=\!50$)
 & ($n\!=\!15,\, m\!=\!25$)
 & ($n\!=\!166,\, m\!=\!3320$) \\
\midrule
EDMD        & \sciRange{1}{2}{-3} & \sciRange{1}{2}{-4} & \sci{2}{-2} \\
F-EDMD      & \sciRange{2}{3}{-3} & \sciRange{8}{9}{-4} & \sciRange{2}{3}{-1} \\
T-SSD       & \sciRange{5}{6}{-3} & \sciRange{5}{8}{-3} & \sciRange{8}{10}{-1} \\
RFB-EDMD    & \sciRange{4}{5}{-3} & \sci{8}{-4}         & \sciRange{1.3}{1.5}{0} \\
ResDMD      & \sci{2}{-3}         & \sciRange{1}{2}{-3} & \sciRange{2}{3}{-1} \\
\bottomrule
\end{tabular}
\caption{Wall-clock times (seconds) for each method on the three 
test problems.}
\label{tab:cost}
\end{table}

\begin{remark}[Complexity]\label{rmk:complexity}
{
The dominant cost of Filtered EDMD is the SVD of the $2r_j\times m$ 
stacked matrix in Algorithm~\ref{alg:iter-intersect}, requiring 
$O(\min(2r_j,m)\cdot 2r_j\cdot m)$ operations at level $j$. Hence the
total intersection cost is
$O\!\left(\sum_j\min(2r_j,m)\cdot 2r_j\cdot m\right)$; in the worst case,
this can exceed the cost of a single EDMD solve. In our examples, the chain
terminates in $1$--$4$ steps, and the reported Filtered EDMD runtime is at most
$0.3$ seconds, although it is several times slower than EDMD on the polynomial
and Van der Pol tests.
}
\end{remark}

\section{Conclusion and future work}\label{sec:conclusion}

{
We introduced the Projected Koopman Operator Approximation (PKOA)
framework for constructing \emph{Filtered EDMD} operators, together
with SVD-based algorithms that realize the coordinate-orthogonal
projection and compute the forward--intersection chain from data.
At the population level, every admissible projection preserves the nonzero
Koopman eigenpairs contained in the dictionary. The experiments show that the
intermediate filtered levels can remove spurious content while
retaining useful structure that is absent from the terminal invariant space.
Together, ordinary EDMD, the intermediate filtered levels, and the terminal
invariant range form a hierarchy of projection ranges. On the Van~der~Pol example,
the spectral content recovered at an intermediate level depends on how the
dictionary aligns with the equilibrium and limit-cycle eigenfunctions.
}

Several directions remain open.
First, while the PKOA framework admits any admissible projection,
the present algorithms implement only the coordinate-orthogonal
choice; characterizing the \emph{optimal} projection within this
family---e.g., one that minimizes spectral pollution or maximizes
spectral recovery for a given dictionary---is a natural theoretical
and computational question.
Second, understanding the robustness of Filtered EDMD under
observation noise and finite-sample effects would strengthen its
practical applicability; preliminary experiments suggest that the
filtering step has a regularizing effect, but a quantitative
analysis remains to be developed.
Finally, applying Filtered EDMD to prediction, control, and
stability analysis could clarify how the one-step compatibility
criterion complements invariance-based methods in practical
workflows.

\section*{Acknowledgments}
S.~Tang was supported by NSF DMS Career Grant No.2340631.

\begin{appendices}

\section{Generalized Eigenfunctions}\label{App:generalized_eigenfunction}

A function $\psi\in L^2(\mu)$ is called a \emph{generalized eigenfunction of order $k$} 
if it satisfies
\[
(K-\lambda I)^k \psi=0
\]
for some Koopman eigenvalue $\lambda$ and some integer $k \geq 1$.  
Ordinary Koopman eigenfunctions correspond to the case $k=1$.  
Any \emph{finite-dimensional} invariant subspace of the Koopman operator 
is spanned by a set of generalized eigenfunctions 
(this follows from the Jordan normal form of $K$ restricted to the subspace).

When the matrix $\mathbf{A}$ in Proposition~\ref{prop:spectral-correspondence} 
is not diagonalizable, the spectral correspondence extends as follows: 
the eigenvalues of $\mathbf{A}$ are still Koopman eigenvalues, and 
the \emph{generalized} left eigenvectors of $\mathbf{A}$ 
(i.e., vectors $\mathbf{l}$ satisfying $(\mathbf{A}^*-\bar\lambda I)^k\mathbf{l}=0$ 
for some $k\ge 1$) yield generalized Koopman eigenfunctions via 
$\psi(\cdot)=\langle \mathbf{g}(\cdot),\mathbf{l}\rangle$.

\section{Koopman spectra associated with the Van der Pol limit cycle and equilibrium}
\label{APP:VDP_Koop}
Let \(\boldsymbol{\Phi}^t\) denote the Van der Pol flow and let
\(K^t f=f\circ \boldsymbol{\Phi}^t\) be the associated Koopman operator. The
same flow has two relevant invariant objects: the attracting limit cycle
\(\Gamma_\mu\) and the unstable equilibrium \(\mathbf{x}_*=0\).  Different
choices of domain and observable space lead to different Koopman point spectra.
First consider the basin \(\mathcal B(\Gamma_\mu)\) of the limit cycle.  Following
the limit-cycle construction in \cite[Sec.~7]{Spectrum}, we use phase-amplitude
coordinates \(\theta:\mathcal B(\Gamma_\mu)\to \mathbb S^1\) and
\(r:\mathcal B(\Gamma_\mu)\to\mathbb R\), satisfying
\[
    \theta(\boldsymbol{\Phi}^t\mathbf{x})=\theta(\mathbf{x})+\omega t,
    \qquad
    r(\boldsymbol{\Phi}^t\mathbf{x})=e^{\sigma t}r(\mathbf{x}),
    \qquad \sigma<0 .
\]
Thus \(\phi_{i\omega}=e^{i\theta}\) and \(\phi_{\sigma}=r\) are Koopman
eigenfunctions:
\[
    K^t\phi_{i\omega}=e^{i\omega t}\phi_{i\omega},
    \qquad
    K^t\phi_{\sigma}=e^{\sigma t}\phi_{\sigma}.
\]
The limit-cycle observable space is
\[
    \mathcal H_{\rm lc}
    =
    \left\{
    f(\mathbf{x})=
    \sum_{\ell\ge 0}\sum_{k\in\mathbb Z}
    a_{\ell k}\,
    \phi_{\sigma}(\mathbf{x})^\ell
    \phi_{i\omega}(\mathbf{x})^k
    :
    \sum_{\ell,k}|a_{\ell k}|^2 w_{\ell k}<\infty
    \right\},
\]
where the weights \(w_{\ell k}>0\) are chosen so that the expansion is analytic in
the transverse coordinate \(\phi_\sigma\) and square-integrable in the phase
coordinate.  On \(\mathcal H_{\rm lc}\),
\[
    K^t\!\left(\phi_{\sigma}^{\ell}\phi_{i\omega}^{k}\right)
    =
    e^{(\ell\sigma+ik\omega)t}\phi_{\sigma}^{\ell}\phi_{i\omega}^{k},
    \qquad
    \Lambda_{\rm lc}
    =
    \left\{
        ik\omega+\ell\sigma
        :
        k\in\mathbb Z,\ \ell\in\mathbb N
    \right\}.
\]
We take
\(\mathbb N=\{0,1,2,\ldots\}\). The corresponding isochrons and isostables are
\[
    \{ \mathbf{x}:\arg\phi_{i\omega}(\mathbf{x})=\mathrm{const}\},
    \qquad
    \{ \mathbf{x}:\phi_{\sigma}(\mathbf{x})=\mathrm{const}\}.
\]
We now consider the unstable equilibrium \(\mathbf{x}_*=0\).  Its Jacobian is
\[
    \begin{pmatrix}
        0 & 1\\
        -1 & \mu
    \end{pmatrix},
    \qquad
    \alpha_{\pm}
    =
    \frac{\mu\pm\sqrt{\mu^2-4}}{2}.
\]
For \(0<\mu<2\), \(\alpha_{\pm}\) are a complex conjugate pair.  In a
linearization neighborhood \(D_{\rm eq}\), Poincare linearization yields local
analytic Koopman eigen-coordinates \(s_+,s_-:D_{\rm eq}\to\mathbb C\), satisfying
\[
    s_{\pm}(\boldsymbol{\Phi}^t\mathbf{x})
    =
    e^{\alpha_{\pm}t}s_{\pm}(\mathbf{x}),
    \qquad
    K^t s_{\pm}=e^{\alpha_{\pm}t}s_{\pm},
\]
for all times for which the trajectory remains in \(D_{\rm eq}\). 
The local
equilibrium observable space is
\[
    \mathcal H_{\rm eq}
    =
    \left\{
    f(\mathbf{x})=
    \sum_{p,q\ge 0}
    b_{pq}\,
    s_+(\mathbf{x})^p s_-(\mathbf{x})^q
    :
    \sum_{p,q}|b_{pq}|^2 \widetilde w_{pq}<\infty
    \right\},
\]
where the weights \(\widetilde w_{pq}>0\) are chosen so that the power series has a
positive radius of convergence in both local eigen-coordinates.  This is in the spirit of
the equilibrium spectral expansion in \cite[Sec.~5]{Spectrum}. 
On this space,
\[
    K^t(s_+^p s_-^q)
    =
    e^{(p\alpha_+ + q\alpha_-)t}s_+^p s_-^q,
    \qquad
    \Lambda_{\rm eq}
    =
    \left\{
        p\alpha_+ + q\alpha_-
        :
        p,q\in\mathbb N
    \right\}.
\]
Writing \(s_+(\mathbf{x})=\rho_{\rm eq}(\mathbf{x})
e^{i\vartheta_{\rm eq}(\mathbf{x})}\), the phase-like and amplitude-like foliations near the unstable equilibrium are
\[
    \{\mathbf{x}:\arg s_+(\mathbf{x})=\mathrm{const}\},
    \qquad
    \{\mathbf{x}:|s_+(\mathbf{x})|=\mathrm{const}\}.
\]
They are not the global isochrons or isostables of the attracting limit cycle, but
they are locally analogous to the isochron/isostable structure associated with
Koopman eigenfunction level sets. 
Hence
\(\Lambda_{\rm lc}\) and \(\Lambda_{\rm eq}\) are spectra of the same Koopman
family \(K^t\), but on different domains and observable spaces:
\(\mathcal H_{\rm lc}\) is adapted to the attracting limit cycle, while
\(\mathcal H_{\rm eq}\) is adapted to the unstable equilibrium.

\section{Auxiliary Result}\label{APP:Auxiliary_Result}

\begin{proposition}[Convergence of the sampled coefficient projectors]
\label{prop:sample-projector-convergence}
Let $\{g_j\}_{j=1}^n\subset L^2(\mu)$ be $L^2(\mu)$-linearly independent, set
$\mathcal G=\mathrm{span}\{g_1,\dots,g_n\}$ and $\mathbf g=(g_1,\dots,g_n)^\top$.
Let $K$ be the Koopman operator on $L^2(\mu)$ and let $P_{\mathcal G}^\mu$ be the 
$L^2(\mu)$-orthogonal projector onto $\mathcal G$.
For i.i.d.\ samples $\{\mathbf x_i\}_{i=1}^m\sim\mu$ define
\[
\mathbf X=\big[\mathbf g(\mathbf x_1)\ \cdots\ \mathbf g(\mathbf x_m)\big]
  \in\mathbb C^{n\times m},\qquad
\mathbf Y=\big[(K\mathbf g)(\mathbf x_1)\ \cdots\ (K\mathbf g)(\mathbf x_m)\big]
  \in\mathbb C^{n\times m}.
\]
Let
\[
W_{\mathcal S_1,m}:=\big\{\mathbf c\in\mathbb C^n:\
  \mathbf c^*\mathbf Y\in\widehat{\mathcal S}_{1,m}\big\},
\qquad
\boldsymbol\Pi_{\mathcal S_1,m}^{\mathrm c}
:=\operatorname{proj}_{W_{\mathcal S_1,m}},
\]
where $\widehat{\mathcal S}_{1,m}=\operatorname{row}(\mathbf X)
\cap\operatorname{row}(\mathbf Y)$ and $\operatorname{proj}$ is the
Euclidean-orthogonal projector in $\mathbb C^n$,
and define the Gram matrices
\begin{equation}\label{eq:gram-matrices}
\mathbf G_{ij}:=\langle g_i,g_j\rangle_\mu,\quad
\mathbf B_{ij}:=\langle Kg_i,g_j\rangle_\mu,\quad
(\mathbf H_K)_{ij}:=\langle Kg_i,Kg_j\rangle_\mu,
\qquad
\mathbf R:=\mathbf H_K-\mathbf B\,\mathbf G^{-1}\mathbf B^*\succeq 0.
\end{equation}
Set
\begin{equation}\label{eq:population-projector}
W_{\mathcal S_1}
:=\big\{\mathbf c\in\mathbb C^n:\ \Psi_K\mathbf c\in\mathcal G\big\}
=\Psi_K^{-1}(\mathcal S_1),
\qquad
\boldsymbol\Pi_{\mathcal S_1}^{\mathrm c}
:=\operatorname{proj}_{W_{\mathcal S_1}}.
\end{equation}
Then, almost surely as $m\to\infty$:
\begin{enumerate}
\item[\emph{(a)}] \emph{(Population identification)}
$\ker(\mathbf R)=W_{\mathcal S_1}$.
\item[\emph{(b)}] \emph{(Sample identification)} There is an almost surely
finite $m_0$ such that, for every $m\geq m_0$,
\[
W_{\mathcal S_1,m}=W_{\mathcal S_1},
\qquad
\boldsymbol\Pi_{\mathcal S_1,m}^{\mathrm c}
=\boldsymbol\Pi_{\mathcal S_1}^{\mathrm c}.
\]
In particular,
$\|\boldsymbol\Pi_{\mathcal S_1,m}^{\mathrm c}
-\boldsymbol\Pi_{\mathcal S_1}^{\mathrm c}\|\to0$.
\end{enumerate}
\end{proposition}

\begin{proof}
\emph{Step 1 (Row/column reduction and PSD form).}
For any matrix $\mathbf M$, 
$\mathbf w\in\mathrm{row}(\mathbf M)
\iff \mathbf w^*\in\mathrm{range}(\mathbf M^*)=(\ker\mathbf M)^\perp$.
Thus
\[
\mathbf c^*\mathbf Y\in\mathrm{row}(\mathbf X)
\iff
\boldsymbol\Pi_{\mathbf X}^{\perp}\mathbf Y^* \mathbf c=\mathbf 0
\iff
\mathbf c\in\ker\!\big(
\mathbf Y\boldsymbol\Pi_{\mathbf X}^{\perp}\mathbf Y^*\big),
\]
where
$\boldsymbol\Pi_{\mathbf X}^{\perp}:=\mathbf I_m-\mathbf X^\dagger\mathbf X$
is the sample-space projector onto $\ker\mathbf X$. The last equivalence
uses the fact that
$\mathbf R_m:=\mathbf Y\boldsymbol\Pi_{\mathbf X}^{\perp}\mathbf Y^*
\succeq0$, so
$\mathbf R_m\mathbf c = 0$ iff 
$\|\boldsymbol\Pi_{\mathbf X}^{\perp}\mathbf Y^*\mathbf c\|^2
= \mathbf c^*\mathbf R_m\mathbf c = 0$.
Since $\mathbf c^*\mathbf Y$ is always a row of $\mathbf Y$ 
(being a linear combination of the rows of $\mathbf Y$), 
the condition $\mathbf c^*\mathbf Y \in \mathrm{row}(\mathbf Y)$ 
is automatically satisfied.
Hence
\[
W_{\mathcal S_1,m}=\ker(\mathbf R_m),
\qquad
\boldsymbol\Pi_{\mathcal S_1,m}^{\mathrm c}
=\operatorname{proj}_{\ker(\mathbf R_m)}.
\]

\smallskip
\emph{Step 2 (LLN limit of $\mathbf R_m/m$).}
Using $\boldsymbol\Pi_{\mathbf X}^{\perp}
=\mathbf I_m-\mathbf X^*(\mathbf X\mathbf X^*)^\dagger\mathbf X$
and the Strong Law of Large Numbers entrywise,
\[
\frac{1}{m}\mathbf X\mathbf X^*\to \mathbf G\succ 0,\quad
\frac{1}{m}\mathbf Y\mathbf X^*\to \mathbf B,\quad
\frac{1}{m}\mathbf Y\mathbf Y^*\to \mathbf H_K
\qquad\text{a.s.}
\]
Since $\mathbf G$ is invertible (by the $L^2(\mu)$-linear independence of the 
dictionary), $\big(\tfrac{1}{m}\mathbf X\mathbf X^*\big)^{-1}\to \mathbf G^{-1}$ 
a.s., and
\[
\frac{1}{m}\mathbf R_m
=
\frac{1}{m}\mathbf Y\mathbf Y^*
-
\Big(\frac{1}{m}\mathbf Y\mathbf X^*\Big)
\Big(\frac{1}{m}\mathbf X\mathbf X^*\Big)^{-1}
\Big(\frac{1}{m}\mathbf X\mathbf Y^*\Big)
\ \xrightarrow[]{a.s.}\
\mathbf R:=\mathbf H_K-\mathbf B\mathbf G^{-1}\mathbf B^*\succeq 0.
\]

\smallskip
\emph{Step 3 (Characterize $\ker\mathbf R$).}
Let
$\Psi_{\mathcal G}:\mathbb C^n\to\mathcal G$ be given by
$\Psi_{\mathcal G}\mathbf c=\langle\mathbf g,\mathbf c\rangle
=\sum_j\overline{c_j}g_j$, and set $v=\Psi_{\mathcal G}\mathbf c$.
Expanding the orthogonal decomposition 
$Kv = P^\mu_{\mathcal G}(Kv) + (I-P^\mu_{\mathcal G})Kv$ 
and using the Gram matrices gives
\[
\|Kv\|_{L^2(\mu)}^2 = \mathbf c^*\mathbf H_K\mathbf c,
\qquad
\|P^\mu_{\mathcal G}(Kv)\|_{L^2(\mu)}^2 
= \mathbf c^*\mathbf B\,\mathbf G^{-1}\mathbf B^*\mathbf c.
\]
Indeed, the second identity follows from the normal equations for the
best approximation: writing
$P^\mu_{\mathcal G}(Kv)=\Psi_{\mathcal G}\mathbf a
=\sum_i\overline{a_i}g_i$ with coefficient vector $\mathbf a$, the
optimality conditions
$\langle Kv - P^\mu_{\mathcal G}(Kv),\, g_i\rangle_\mu = 0$ yield 
$\mathbf G\,\mathbf a = \mathbf B^*\mathbf c$, 
so $\mathbf a = \mathbf G^{-1}\mathbf B^*\mathbf c$ and 
$\|P^\mu_{\mathcal G}(Kv)\|^2 
= \mathbf a^*\mathbf G\,\mathbf a 
= \mathbf c^*\mathbf B\,\mathbf G^{-1}\mathbf B^*\mathbf c$.
By the Pythagorean theorem,
\[
\|(I-P^\mu_{\mathcal G})Kv\|_{L^2(\mu)}^2
= \mathbf c^*\mathbf H_K\mathbf c
  - \mathbf c^*\mathbf B\,\mathbf G^{-1}\mathbf B^*\mathbf c
= \mathbf c^*\mathbf R\,\mathbf c.
\]
Therefore $\mathbf c\in\ker\mathbf R 
\iff \mathbf c^*\mathbf R\,\mathbf c = 0 
\iff (I-P^\mu_{\mathcal G})\Psi_K\mathbf c=0
\iff \Psi_K\mathbf c\in\mathcal G$,
i.e.\ $\ker\mathbf R=W_{\mathcal S_1}$. This proves~(a).

\smallskip
\emph{Step 4 (Eventual sample identification).}
Choose a basis $\{\mathbf c_1,\ldots,\mathbf c_d\}$ of
$W_{\mathcal S_1}$. For each $\ell$ there is an
$\mathbf a_\ell\in\mathbb C^n$ such that
$\Psi_K\mathbf c_\ell=\Psi_{\mathcal G}\mathbf a_\ell$ in $L^2(\mu)$.
With probability one, all sampled points avoid the union of the corresponding
finite collection of null sets. On this event,
$\mathbf c_\ell^*\mathbf Y=\mathbf a_\ell^*\mathbf X$ for every $m$ and
every $\ell$, so
\[
W_{\mathcal S_1}\subseteq W_{\mathcal S_1,m}=\ker\mathbf R_m
\qquad\text{for every }m.
\]
If $\mathbf R=0$, then $W_{\mathcal S_1}=\mathbb C^n$, and equality follows
immediately. Otherwise, let $\gamma>0$ be the smallest eigenvalue of
$\mathbf R$ on $W_{\mathcal S_1}^{\perp}$. By Step~2, almost surely
$\|\mathbf R_m/m-\mathbf R\|<\gamma/2$ for all sufficiently large $m$.
Hence, for every $\mathbf z\in W_{\mathcal S_1}^{\perp}$,
\[
\mathbf z^*\frac{\mathbf R_m}{m}\mathbf z
\geq\frac{\gamma}{2}\|\mathbf z\|_2^2.
\]
Since $\mathbf R_m$ annihilates $W_{\mathcal S_1}$, this inequality rules
out any additional kernel direction. Thus
$\ker\mathbf R_m=W_{\mathcal S_1}$ for all sufficiently large $m$, proving
the two identities and the convergence in~(b).
\end{proof}

\end{appendices}

\bibliography{references}

@article{RowleyMezicEtAl2009,
  author    = {Rowley, Clarence W. and Mezi{\'c}, Igor and Bagheri, Shervin
               and Schlatter, Philipp and Henningson, Dan S.},
  title     = {Spectral analysis of nonlinear flows},
  journal   = {Journal of Fluid Mechanics},
  volume    = {641},
  pages     = {115--127},
  year      = {2009},
  doi       = {10.1017/S0022112009992059}
}

@article{BudisicMohrMezic2012,
  author    = {Budi{\v{s}}i{\'c}, Marko and Mohr, Ryan and Mezi{\'c}, Igor},
  title     = {Applied {Koopman}ism},
  journal   = {Chaos},
  volume    = {22},
  number    = {4},
  pages     = {047510},
  year      = {2012},
  doi       = {10.1063/1.4772195}
}

@article{williams2015data,
  title={A data--driven approximation of the {Koopman} operator: Extending dynamic mode decomposition},
  author={Williams, Matthew O and Kevrekidis, Ioannis G and Rowley, Clarence W},
  journal={Journal of Nonlinear Science},
  volume={25},
  number={6},
  pages={1307--1346},
  year={2015},
  publisher={Springer}
}

@article{colbrook2024rigorous,
  title={Rigorous data-driven computation of spectral properties of {Koopman} operators for dynamical systems},
  author={Colbrook, Matthew J and Townsend, Alex},
  journal={Communications on Pure and Applied Mathematics},
  volume={77},
  number={1},
  pages={221--283},
  year={2024},
  publisher={Wiley Online Library}
}

@article{Koopman_Form,
title = {Finite dimensional {Koopman} form of polynomial nonlinear systems},
journal = {IFAC-PapersOnLine},
volume = {56},
number = {2},
pages = {6423-6428},
year = {2023},
note = {22nd IFAC World Congress},
issn = {2405-8963},
doi = {https://doi.org/10.1016/j.ifacol.2023.10.849},
url = {https://www.sciencedirect.com/science/article/pii/S2405896323012296},
author = {Lucian C. Iacob and Maarten Schoukens and Roland Tóth},
}

@article{Spectrum,
       author = {Mezi{\'c}, Igor},
        title = {Spectrum of the {Koopman} operator, spectral expansions in functional spaces, and state-space geometry},
      journal = {Journal of NonLinear Science},
         year = 2020,
        month = oct,
       volume = {30},
       number = {5},
        pages = {2091-2145},
          doi = {10.1007/s00332-019-09598-5},
archivePrefix = {arXiv},
       eprint = {1702.07597},
 primaryClass = {nlin.CD},
       adsurl = {https://ui.adsabs.harvard.edu/abs/2020JNS....30.2091M}
}

@article{koopman1931hamiltonian,
  title={Hamiltonian systems and transformation in {Hilbert} space},
  author={Koopman, Bernard O},
  journal={Proceedings of the National Academy of Sciences},
  volume={17},
  number={5},
  pages={315--318},
  year={1931}
}

@article{koopman1932dynamical,
  title={Dynamical systems of continuous spectra},
  author={Koopman, Bernard O and Neumann, J v},
  journal={Proceedings of the National Academy of Sciences},
  volume={18},
  number={3},
  pages={255--263},
  year={1932}
}

@article{AMSNoticesKoopman,
  title={{Koopman} operator, geometry, and learning of dynamical systems},
  author={Mezi{\'c}, Igor},
  journal={Notices of the American Mathematical Society},
  year         = {2021},
month = {8},
pages = {1087-1105},
  volume       = {68},
  number       = {7},

}

@article{KordaMezic2018,
  title={On convergence of extended dynamic mode decomposition to the {Koopman} operator},
  author={Milan Korda and Igor Mezi{\'c}},
  journal={Journal of Nonlinear Science},
  year={2018},
  volume={28},
  number={2},
  pages={687--710},
  doi={10.1007/s00332-017-9423-0}
}

@article{Mauroy2018,
  title={Global computation of phase–amplitude reduction for limit-cycle dynamics},
  author={Mauroy, Alexandre and Mezi{\'c}, Igor},
  journal={Chaos},
  year={2018},
  volume={28},
  number={7},
  pages={073108},
  doi={10.1063/1.5030175}
}

@article{RFB-EDMD,
  author    = {Haseli, Masih and Cort{\'e}s, Jorge},
  title     = {Recursive forward-backward {EDMD}: Guaranteed algebraic search for {Koopman} invariant subspaces},
  journal   = {IEEE Access},
  year      = {2025},
  volume    = {13},
  pages     = {61 006--61 025},
}

@article{haseli2023generalizing,
  title={Generalizing dynamic mode decomposition: Balancing accuracy and expressiveness in {Koopman} approximations},
  author={Haseli, Masih and Cort{\'e}s, Jorge},
  journal={Automatica},
  volume={153},
  pages={111001},
  year={2023},
  publisher={Elsevier}
}

@article{haseli2021learning,
  title={Learning {Koopman} eigenfunctions and invariant subspaces from data: Symmetric subspace decomposition},
  author={Haseli, Masih and Cort{\'e}s, Jorge},
  journal={IEEE Transactions on Automatic Control},
  volume={67},
  number={7},
  pages={3442--3457},
  year={2022},
  publisher={IEEE}
}

@article{colbrook2023mpedmd,
  title     = {The {mpEDMD} algorithm for data-driven computations of
               measure-preserving dynamical systems},
  author    = {Colbrook, Matthew J.},
  journal   = {SIAM Journal on Numerical Analysis},
  volume    = {61},
  number    = {3},
  pages     = {1585--1608},
  year      = {2023},
  publisher = {SIAM},
  doi       = {10.1137/22M1521407}
}

@article{colbrook2025rigged,
  title     = {Rigged dynamic mode decomposition: Data-driven generalized
               eigenfunction decompositions for {Koopman} operators},
  author    = {Colbrook, Matthew J. and Drysdale, Catherine and Horning, Andrew},
  journal   = {SIAM Journal on Applied Dynamical Systems},
  volume    = {24},
  number    = {2},
  pages     = {1150--1190},
  year      = {2025},
  publisher = {SIAM},
  doi       = {10.1137/24M1662370}
}

@article{colbrook2023resdmd_fluid,
  title     = {Residual dynamic mode decomposition: Robust and verified
               {Koopman}ism},
  author    = {Colbrook, Matthew J. and Ayton, Lorna J. and Sz\H{o}ke, Mih\'aly},
  journal   = {Journal of Fluid Mechanics},
  volume    = {955},
  pages     = {A21},
  year      = {2023},
  publisher = {Cambridge University Press},
  doi       = {10.1017/jfm.2022.1052}
}

@misc{mauroy2024analytic,
  title     = {Analytic extended dynamic mode decomposition},
  author    = {Mauroy, Alexandre and Mezi{\'c}, Igor},
  year      = {2024},
  note      = {Preprint, arXiv:2405.15945},
  eprint    = {2405.15945},
  archiveprefix = {arXiv},
  primaryclass  = {math.DS}
}

@article{philipp2025kernel_edmd,
  title     = {Error analysis of kernel {EDMD} for prediction and control in the {Koopman} framework},
  author    = {Philipp, Friedrich M. and Schaller, Manuel and Worthmann, Karl and Peitz, Sebastian and N{\"u}ske, Feliks},
  journal   = {Journal of Nonlinear Science},
  volume    = {35},
  pages     = {92},
  year      = {2025},
  publisher = {Springer},
  doi       = {10.1007/s00332-025-10182-3}
}

@article{kohne2025kernel_linf,
  title     = {{$L^{\infty}$}-error bounds for approximations of the {Koopman} operator by kernel extended dynamic mode decomposition},
  author    = {K{\"o}hne, Frederik and Philipp, Friedrich M. and Schaller, Manuel and Schiela, Anton and Worthmann, Karl},
  journal   = {SIAM Journal on Applied Dynamical Systems},
  volume    = {24},
  number    = {1},
  pages     = {501--529},
  year      = {2025},
  publisher = {SIAM},
  doi       = {10.1137/24M1650120}
}

@incollection{colbrook2024multiverse,
  title     = {The multiverse of dynamic mode decomposition algorithms},
  author    = {Colbrook, Matthew J.},
  booktitle = {Handbook of Numerical Analysis},
  volume    = {25},
  pages     = {127--230},
  year      = {2024},
  publisher = {Elsevier},
  doi       = {10.1016/bs.hna.2024.05.004}
}

@article{Mezic2005,
  author    = {Mezi{\'c}, Igor},
  title     = {Spectral properties of dynamical systems, model
               reduction and decompositions},
  journal   = {Nonlinear Dynamics},
  volume    = {41},
  number    = {1--3},
  pages     = {309--325},
  year      = {2005},
  doi       = {10.1007/s11071-005-2824-x},
}

@article{Mezic2013,
  author  = {Mezi{\'c}, Igor},
  title   = {Analysis of fluid flows via spectral properties of the {Koopman} operator},
  journal = {Annual Review of Fluid Mechanics},
  volume  = {45},
  number  = {1},
  pages   = {357--378},
  year    = {2013},
  doi     = {10.1146/annurev-fluid-011212-140652},
}

@article{KernelEDMD,
  author    = {Williams, Matthew O. and Rowley, Clarence W.
               and Kevrekidis, Ioannis G.},
  title     = {A kernel-based method for data-driven {Koopman}
               spectral analysis},
  journal   = {Journal of Computational Dynamics},
  volume    = {2},
  number    = {2},
  pages     = {247--265},
  year      = {2015},
  doi       = {10.3934/jcd.2015005},
}

@article{OttoRowley2019,
  author    = {Otto, Samuel E. and Rowley, Clarence W.},
  title     = {Linearly-recurrent autoencoder networks for learning dynamics},
  journal   = {SIAM Journal on Applied Dynamical Systems},
  volume    = {18},
  number    = {1},
  pages     = {558--593},
  year      = {2019},
  doi       = {10.1137/18M1177846},
}

@article{OttoMacchioRowley2023,
  author    = {Otto, Samuel E. and Macchio, Gregory R. and Rowley, Clarence W.},
  title     = {Learning nonlinear projections for reduced-order modeling of dynamical systems using constrained autoencoders},
  journal   = {Chaos: An Interdisciplinary Journal of Nonlinear Science},
  volume    = {33},
  number    = {11},
  pages     = {113130},
  year      = {2023},
  doi       = {10.1063/5.0169688},
}

@article{DeepKoopman,
  author    = {Lusch, Bethany and Kutz, J. Nathan
               and Brunton, Steven L.},
  title     = {Deep learning for universal linear embeddings
               of nonlinear dynamics},
  journal   = {Nature Communications},
  volume    = {9},
  number    = {1},
  pages     = {4950},
  year      = {2018},
  doi       = {10.1038/s41467-018-07210-0},
}

@article{conradie2026trustworthy,
  title={Trustworthy {Koopman} Operator Learning: Invariance Diagnostics and Error Bounds},
  author={Conradie, Gustav and Boull{\'e}, Nicolas and Loiseau, Jean-Christophe
          and Brunton, Steven L. and Colbrook, Matthew J.},
  journal={arXiv preprint arXiv:2603.15091},
  year={2026}
}

@article{korda2020data,
  title={Data-driven spectral analysis of the {Koopman} operator},
  author={Korda, Milan and Putinar, Mihai and Mezi{\'c}, Igor},
  journal={Applied and Computational Harmonic Analysis},
  volume={48},
  number={2},
  pages={599--629},
  year={2020}
}

@article{das2019delay,
  title={Delay-coordinate maps and the spectra of {Koopman} operators},
  author={Das, Suddhasattwa and Giannakis, Dimitrios},
  journal={Journal of Statistical Physics},
  volume={175},
  number={6},
  pages={1107--1145},
  year={2019}
}

@article{Mezic2022Numerical,
  author  = {Mezi{\'c}, Igor},
  title   = {On Numerical Approximations of the {Koopman} Operator},
  journal = {Mathematics},
  volume  = {10},
  number  = {7},
  pages   = {1180},
  year    = {2022}
}

@article{drmac2024koopmanschur,
  title={A data driven {Koopman}--{Schur} decomposition for computational analysis of nonlinear dynamics},
  author={Drma{\v c}, Zlatko and Mezi{\'c}, Igor},
  journal={arXiv preprint arXiv:2312.15837},
  year={2024}
}

\end{document}